\documentclass[a4paper,11pt]{amsart}
\usepackage{amssymb,amscd}
\usepackage{enumitem}
\usepackage[capitalize]{cleveref}
\newcounter{alph}

\newtheorem{theo}[alph]{Theorem}
\numberwithin{equation}{section}
\newtheorem{cor}[equation]{Corollary}
\newtheorem{lem}[equation]{Lemma}
\newtheorem{prop}[equation]{Proposition}
\newtheorem{thm}[equation]{Theorem}

\theoremstyle{definition}

\newtheorem{exa}[equation]{Example}

\newtheorem{que}[equation]{Question}
\newtheorem{rem}[equation]{Remark}

\def\C{\mathbb C}
\def\H{\mathbb H}
\def\N{\mathbb N}
\def\R{\mathbb R}
\def\O{\mathbb O}
\def\ve{\varepsilon}
\def\vf{\varphi}
\def\la{\langle}
\def\ra{\rangle}
\newcommand{\ad}{\operatorname{ad}}
\newcommand{\ess}{\operatorname{ess}}
\newcommand{\diam}{\operatorname{diam}}
\newcommand{\dive}{\operatorname{div}}
\newcommand{\dv}{\operatorname{dv}}
\newcommand{\id}{\operatorname{id}}

\newcommand{\grad}{\operatorname{grad}}

\newcommand{\supp}{\operatorname{supp}}

\begin{document}


\title{Bottom of spectra and coverings}
\author{Werner Ballmann}
\address
{WB: Max Planck Institute for Mathematics,
Vivatsgasse 7, 53111 Bonn}
\email{hwbllmnn\@@mpim-bonn.mpg.de}
\author{Panagiotis Polymerakis}
\address{PP: Max Planck Institute for Mathematics,
Vivatsgasse 7, 53111 Bonn}
\email{polymerp\@@mpim-bonn.mpg.de}


\thanks{\emph{Acknowledgments.}
We are grateful to the Max Planck Institute for Mathematics and the Hausdorff Center for Mathematics in Bonn for their support and hospitality.}

\date{\today}


\begin{abstract}
We discuss the behaviour of the bottom of the spectrum of scalar Schr\"odinger operators
under Riemannian coverings.
\end{abstract}

\maketitle

\tableofcontents

\section{Introduction}
\label{intro}

The spectrum $\sigma(M)$ of a Riemannian manifold $M$
is an interesting geometric invariant,
and the relation of $\sigma(M)$ to other geometric invariants of $M$ has attracted much attention.  
We are interested in the behaviour of the bottom $\lambda_0(M)=\inf\sigma(M)$ of the spectrum of $M$
under Riemannian coverings.
More generally,
we study the behaviour of the bottom of the spectrum of (scalar) Schr\"odinger operators under coverings.
Here a Schr\"odinger operator on $M$ is an operator of the form
\begin{align}
	S = \Delta +V,
\end{align}
where $\Delta$ denotes the Laplacian of $M$ and $V\in C^\infty(M)$.
Then $S$ with domain $C^\infty_c(M)\subseteq L^2(M)$ is a symmetric operator.

\emph{We assume throughout that $S$ is bounded from below} (on $C^\infty_c(M)$).
Then the Friedrichs extension $\bar S$ of $S$ exists and is a self-adjoint operator.
If $M$ is complete,
then $S$ is essentially self-adjoint, that is, the closure of $S$ coincides with $\bar S$.
If $M$ is isometric to the interior of a complete Riemannian manifold with boundary,
then $\bar S$ coincides with the extension of $S$
associated to the Dirichlet boundary condition.

Recall that $\Delta$ is non-negative, and hence $S$ is bounded from below if $V$ is bounded from below.

We denote the spectrum and the essential spectrum of $\bar S$
by $\sigma(S,M)$ and $\sigma_{\ess}(S,M)$, respectively.
In the case of the Laplacian, we also write $\sigma(M)$ 
--this is what was meant by the spectrum of $M$ in the first paragraph-- and $\sigma_{\ess}(M)$.
Recall that, for a Lipschitz function $f\ne0$ on $M$ with compact support,
\begin{align}\label{raylei}
  R_S(f) = \frac{\int_M (| \grad f|^2+Vf^2)}{\int_Mf^2}
\end{align}
is called the \emph{Rayleigh quotient} of $f$ (with respect to $S$).
It is important that the \emph{bottom $\inf\sigma(S,M)$ of the spectrum of $\bar S$} is given by
\begin{align}\label{bottspec}
	\lambda_0(S,M) = \inf R_S(f),
\end{align}
where the infimum is taken over all non-zero $f\in C^\infty_c(M)$,
or, equivalently, over all non-zero Lipschitz functions $f$ on $M$ with compact support.
The \emph{bottom $\inf\sigma_{\ess}(S,M)$ of the essential spectrum of $\bar S$} is given by
\begin{align*}
	\lambda_{\ess}(S,M) = \sup \lambda_0(S,M\setminus K),
\end{align*}
where the supremum is taken over all compact subsets $K$ of $M$.
(This is well known in the case where $M$ is complete.)
In the case of the Laplacian, $S=\Delta$, we also write $\lambda_0(M)$ and $\lambda_{\ess}(M)$.

Consider now a Riemannian covering $p\colon M_1\to M_0$,
a Schr\"odinger operator $S_0$ on $M_0$, and its lift $S_1$ under $p$ to $M_1$.
We assume for now that $M_0$ and $M_1$ are connected,
although it is important in intermediate steps of our later discussion that $M_1$ may also not be connected.
We denote by $\Gamma$ the group of covering transformations of $p$.
It is transitive on the fibers of $p$ if and only if $p$ is a normal covering.
In the short \cref{secele}, we discuss the following general inequality.

\begin{theo}[Monotonicity]\label{monot}
We always have $\lambda_{0}(S_1,M_{1})\ge\lambda_{0}(S_0,M_{0})$.
\end{theo}

We say that $p$ is \emph{amenable} if the right action of the fundamental group $\pi_1(M_0,x)$
on the fiber $p^{-1}(x)$ is amenable for one or, equivalently, for any $x\in M_0$ (see \cref{secamen}).
If $p$ is normal,
then $p$ is amenable if and only if the group $\Gamma$ of covering transformations of $p$ is amenable.

\begin{theo}\label{tame}
If $p$ is amenable, then $\lambda_0(S_1,M_1)=\lambda_0(S_0,M_0)$.
\end{theo}

In \cref{secame}, we will review the evolution of \cref{tame} from first versions in Brooks's \cite{Br1,Br2}
over intermediate versions in articles by Berard-Castillon \cite{BC} and Ji-Li-Wang \cite{JLW}
to the above version from our article \cite{BMP1}.

The problem of when the monotonicity $\lambda_{0}(S_1,M_{1})\ge\lambda_{0}(S_0,M_{0})$ is strict is much more intricate.

\begin{theo}\label{name}
If $p$ is not amenable and $\lambda_{\ess}(S_0,M_0)>\lambda_0(S_0,M_0)$,
then $\lambda_0(S_1,M_1)>\lambda_0(S_0,M_0)$.
\end{theo}

We have $\lambda_{\ess}(S_0,M_0)=\infty$ in the case where the base manifold is closed,
so that we have a strict equivalence between amenability of $p$ and the equality $\lambda_0(S_1,M_1)=\lambda_0(S_0,M_0)$ in this case.
This equivalence, for the Laplacian in the case of the universal covering of a closed manifold, is the content of Brooks's article \cite{Br1}.
In the later article \cite{Br2},
he extended his result to normal coverings $p$,
where $M_0$ is noncompact of finite topological type (in his sense)
and where the covering $p$ admits a fundamental domain
which satisfies a certain isoperimetric inequality.
Roblin and Tapie \cite{RT} obtained an analogous result under a different requirement,
namely the existence of a spectrally optimal fundamental domain (in their sense)
and an assumption on its Neumann spectrum.
In \cite{Br2}, Brooks speculates whether his results hold
under the more general condition that $\lambda_{\ess}(M_0)>\lambda_0(M_0)$,
not assuming the existence of any kind of specific fundamental domains.
Under the assumption that the Ricci curvature of $M_0$ is bounded from below,
this is the main result in \cite{BMP2}.
Finally, the second author of this article established \cref{name} in full generality \cite{Po2}.
We review this development in more detail in \cref{susname}.

\begin{rem}\label{sullivan}
Let $\lambda_0=\lambda_0(M)$
and $p_t(x,y)$ be the heat kernel associated to the Friedrichs extension of $\Delta$.
Following Sullivan \cite{Su}, we say that $\lambda\in\R$ belongs to the \emph{Green's region} of $M$ if
\begin{align*}
	G_\lambda(x,y) = \int_0^\infty e^{\lambda t}p_t(x,y)dt < \infty
\end{align*}
for some (and then all) $x\ne y$ in $M$.
In \cite[Theorem 2.6]{Su},
Sullivan establishes that the Green's region is either $(-\infty,\lambda_0)$ or else  $(-\infty,\lambda_0]$,
and calls $M$ then \emph{$\lambda_0$-recurrent} and \emph{$\lambda_0$-transient}, respectively.
In the $\lambda_0$-recurrent case,
positive $\lambda_0$-harmonic functions on $M$ are constant multiples of one another,
and the random process on $M$ with transition densities
\begin{align*}
	e^{\lambda_0t}p_t(x,y)\vf(y)/\vf(x)
\end{align*}
is recurrent, where $\vf$ is a positive $\lambda_0$-harmonic function on $M$ \cite[Theorems 2.7 and 2.10]{Su}.
If $\lambda_{\ess}(M)>\lambda_0$, then $M$ is $\lambda_0$-recurrent \cite[Theorem 2.8]{Su}.
More generally, if $\lambda_0$ is an eigenvalue of $M$,
then $M$ is $\lambda_0$-recurrent.
In view of this, it is natural to ask whether the assertion of \cref{name} holds (for the Laplacian)
under the weaker assumption that $M_0$ is $\lambda_0(M_0)$-recurrent and $p$ is not amenable
or that $\lambda_0(M_0)$ is an eigenvalue of $M_0$ and $p$ is not amenable.
However, there are counterexamples, even to the latter question.
Namely, for $0<\alpha<1$, consider a closed surface $S_\alpha$
with, say, one puncture and a Riemannian metric on it such that, around the puncture,
it is isometric to the surface of revolution with profile curve $(x,\exp(-x^\alpha))_{x\ge1}$.
Brooks \cite[Section 1]{Br2} obtained that $S_\alpha$ is complete with finite volume,
and hence $\lambda_0(S_\alpha)=0$ is an eigenvalue of $S_\alpha$ with eigenfunction $1$,
and that the bottom of the spectrum of the universal covering of $S_\alpha$ is equal to zero \cite[Lemma 1]{Br2}.
Moreover, if the Euler number of $S_\alpha$ is negative,
then its fundamental group $\Gamma$ contains the free group in two generators as a subgroup.
In particular, $\Gamma$ is non-amenable and, hence,
the universal covering of $S_\alpha$ and many Riemannian coverings of $S_\alpha$ in between are counterexamples.

Brooks also explains a hyperbolic counterexample \cite[Section 1]{Br2};
compare with \cref{sechype} and, in particular, \cref{remhype} below.
\end{rem}

\begin{que}\label{question}
Is there a reasonable replacement of the assumption
that $\lambda_0(S_0,M_0)$ does not belong to $\sigma_{\ess}(S_{0},M_{0})$ in \cref{name},
which generalizes \cref{name} or even turns the conclusion into an equivalence?
\end{que}

\noindent{\bf Structure of the article.}
In the second section, we introduce some notation,
discuss the renormalization of scalar Laplace type operators,
and introduce the concept of the amenability of actions of countable groups on countable sets.
In the ensuing three sections, we discuss Theorems \ref{monot} -- \ref{name}.
Whereas the proofs of Theorems \ref{monot} and \ref{tame} seem quite satisfactory,
we indicate a simplification of the existing proof of \cref{name} in \cref{susim}.
\cref{sechype} contains a new application to geometrically finite locally symmetric spaces of negative sectional curvature.
In the appendix, we discuss some basic relations between geometry and the analysis of differential operators.
Since it does not make an essential difference in the lines of arguments,
we consider differential operators on Hermitian or Riemannian vector bundles
over Riemannian manifolds $M$ (with weighted measures).
Under natural assumptions,
we establish essential self-adjointness in the case where $M$ is complete,
the characterization of the essential spectrum by geometric Weyl sequences,
and the stability of the essential spectrum under removal of compact domains.
These latter properties are known in the case where $M$ is complete,
but we could not ferret out a reference
for the case of the Friedrichs extension of the operator in the incomplete case.
Since this case has some subtleties, the inclusion of the discussion seems justified.

\section{Preliminaries}
\label{secpre}

We let $M$ be a Riemannian manifold of dimension $m$.
For a Borel subset $A$ of $M$,
we denote by $|A|$ the volume of $A$ with respect to the volume element $\dv$ of $M$.
Similarly, for a submanifold $N$ of $M$ of dimension $n<m$ and a Borel subset $B$ of $N$,
we let $|B|$ be the $n$-dimensional Riemannian volume of $B$.
To avoid confusion, we also write $|B|_n$ if necessary. 
We call
\begin{align}\label{checon}
 	h(M) = \inf\frac{|\partial D|}{|D|}
 	= \inf\frac{|\partial D|_{m-1}}{|D|_m}
 	\hspace{3mm}\text{and}\hspace{3mm}
	h_{\ess}(M) = \sup h(M\setminus K)
\end{align}
the \emph{Cheeger constant} and \emph{asymptotic Cheeger constant of $M$}, respectively.
Here the infimum is taken over all compact domains $D\subseteq M$ with smooth boundary $\partial D$
and the supremum over all compact subsets $K$ of $M$.
The respective \emph{Cheeger inequality} asserts that
\begin{align}\label{chein}
	\lambda_0(M) \ge \frac14 h^2(M)
	\hspace{3mm}\text{and}\hspace{3mm}
	\lambda_{\ess}(M) \ge \frac14 h_{\ess}^2(M).
\end{align}
Frequently, we consider a \emph{weighted measure} on $M$, that is,
a measure of the form $\vf^2\dv$, where $\vf\in C^\infty(M)$ is positive.
For a Borel subset $A$ of $M$, we then denote by $|A|_\vf$ the $\vf$-volume of $A$,
\begin{align}\label{phivol}
  |A|_\vf = \int_A \vf^2 = \int_A \vf^2\dv.
\end{align}
Similarly, for a submanifold $N$ of $M$ of dimension $n<m$ and a Borel subset $B$ of $N$,
we let $|B|_\vf$ be the $n$-dimensional $\vf$-volume of $B$.
To avoid confusion, we also write $|B|_{n,\vf}$ if necessary. 

We write $L^2(M)$ or $L^2(M,\dv)$ for the space of equivalence classes of
measurable functions on $M$ which are square-integrable with respect to $\dv$.
Similarly, if $\mu=\vf^2\dv$ is a weighted measure on $M$,
we write $L^2(M,\vf^2\dv)$ or $L^2(M,\mu)$ for the space of equivalence classes of
measurable functions on $M$ which are square-integrable with respect to $\mu$.

\subsection{Renormalizing scalar operators}
\label{suseren}
The idea of renormalizing the Lap\-lacian occurs e.g.\ in \cite[Section 2]{Br2} and \cite[Section 8]{Su}.
The idea works as well for Schr\"odinger operators, as explained in \cite[Section 7]{Po};
compare also with \cite[Section 7]{Co}.

We consider the following four types of \emph{scalar differential operators},
that is, differential operators on $C^\infty(M)$:
\begin{enumerate}
\item
the \emph{Laplacian} $\Delta$,
\item
\emph{Schr\"odinger operators} $L=\Delta+V$ with \emph{potential} $V$,
\item
\emph{diffusion operators} $L=\Delta+X$ with \emph{drift} $X$,
\item
\emph{Laplace type operators} $L=\Delta+X+V$,
\end{enumerate}
where $V$ is a smooth function and $X$ a smooth vector field on $M$.
The Laplacian and Schr\"odinger operators are formally self-adjoint,
that is, are symmetric on $C^\infty_c(M)\subseteq L^2(M,\dv)$.
Using integration by parts,
a straightforward computation shows the following

\begin{prop}\label{mux}
A diffusion or Laplace type operator $L$ is
\emph{formally self-adjoint with respect to a weighted measure $\mu=\vf^2\dv$ on $M$},
that is, is symmetric on $C^\infty_c(M)\subseteq L^2(M,\mu)$,
if and only if $X=-2\grad\ln\vf$.
\end{prop}

Obviously, any of the above kind of operators is of Laplace type.
Fix a weighted measure $\mu=\vf^2\dv$.
Then
\begin{align}\label{multphi}
	m_\vf \colon L^2(M,\dv) \to L^2(M,\mu), \quad m_\vf f = \vf f,
\end{align}
is an orthogonal transformation (explaining the square of $\vf$ as a weight).

\begin{prop}[Renormalization]\label{renor}
If a Laplace type operator $L$ as above is formally self-adjoint with respect to $\mu$,
then $m_\vf$ intertwines $L$ with the Schr\"odinger operator $S=\Delta+(V-\Delta\vf/\vf)$,
\begin{align*}
	L = m_\vf^{-1} \circ S \circ m_\vf.
\end{align*}
Conversely, for a Schr\"odinger operator $S$ with potential written in the form $V-\Delta\vf/\vf$,
$L$ is of Laplace type and is formally self-adjoint with respect to $\mu$.
In particular, $L$ is non-negative on $C^\infty_c(M)\subseteq L^2(M,\mu)$
if and only if $S$ is non-negative on $C^\infty_c(M)\subseteq L^2(M,\dv)$.
\end{prop}

\begin{proof}
Using \cref{mux}, we have
\begin{align*}
	(\Delta + &(V-\Delta\vf/\vf))(\vf f) \\
	&= \vf\Delta f + f\Delta\vf - 2\la\grad\vf,\grad f\ra + (V-\Delta\vf/\vf)\vf f \\
	&= \vf(\Delta f - 2\la\grad\ln\vf,\grad f\ra +Vf)
	= \vf Lf.
\end{align*}
for any $f\in C^\infty(M)$.
\end{proof}

For $L$ and $S$ as in \cref{renor},
we get that $L$ is bounded from below (on $C^\infty_c(M)\subseteq L^2(M,\mu)$)
if and only if $S$ is bounded from below (on $C^\infty_c(M)\subseteq L^2(M,\dv)$).
So far, lower boundedness did not play a role in this section,
but we assume it from now on (as agreed upon in the introduction).
Then the Friedrichs extensions $\bar L$ and $\bar S$ of $L$ and $S$ are also intertwined by $m_\vf$,
\begin{align*}
	\bar{L} = m^{-1}_{\varphi} \circ \bar{S} \circ m_{\varphi}.
\end{align*}
In particular, we have
\begin{align}
  \sigma(L,M) = \sigma(S,M).
\end{align}
Therefore the spectral theory of Laplace type operators,
which are formally self-adjoint with respect to a weighted measure,
is the same as that of Schr\"odinger operators.
Thus with respect to spectral theory, one could restrict attention to the latter class of operators.
However, renormalization comes in in a second essential way through \eqref{bottom2} below.

Let $S=\Delta+V$ be a Schr\"odinger operator.
We say that a smooth function $\vf$ on $M$ (not necessarily square-integrable)
is \emph{$\lambda$-harmonic (with respect to $S$)} if it solves $S\vf=\lambda\vf$.
Recall from \eqref{bottspec} that $\lambda_0(S,M)$ is given by an infimum over Rayleigh quotients.
At the same time, $\lambda_0(S,M)$ is the maximal $\lambda\in\R$
such that there is a positive $\lambda$-harmonic function on $M$ (\cref{poslam}).

We fix such a $\lambda\le\lambda_0(S,M)$ and positive $\lambda$-harmonic function $\vf$ on $M$.
By \cref{renor}, $m_\vf$ intertwines $S_\vf=S-\lambda$ with the diffusion operator $L=\Delta-2\grad\ln\vf$.
Point one for renormalizing with $\vf$ is that
\begin{align}\label{bottom2}
	\lambda_0(S_\vf,M) = \lambda_0(S,M) - \lambda = \inf\frac{\int_M |\grad f|^2d\mu}{\int_Mf^2d\mu},
\end{align}
where the infimum is taken over all non-vanishing $f\in C^\infty_c(M)$ or, equivalently,
over all Lipschitz functions $f\ne0$ on $M$ with compact support;
see \cite[Section 2]{Br2} for the Laplacian
and \cite[Proposition 7.3]{Po} for the more general case of Schr\"odinger operators.

The \emph{modified Cheeger constant} and \emph{modified asymptotic Cheeger constant of $M$} are given by
\begin{align}\label{modche}
  h_\vf(M) = \inf \frac{|\partial D|_\vf}{|D|_\vf}
  \hspace{3mm}\text{and}\hspace{3mm}
  h_{\vf,\ess}(M) = \sup h_\vf(M\setminus K),
\end{align}
respectively,
where the infimum is taken over all compact domains $D\subseteq M$ with smooth boundary $\partial D$
and the supremum over all compact subsets $K\subseteq M$.
By \cite[Corollaries 7.4 and 7.5]{Po}, we have the \emph{modified Cheeger inequalities}
\begin{align}\label{modche2}
  \lambda_0(S,M) - \lambda \ge h_\vf(M)^2/4
  \hspace{3mm}\text{and}\hspace{3mm}
  \lambda_{\ess}(S,M) - \lambda \ge h_{\vf,\ess}(M)^2/4.
\end{align}
Point two for renormalizing with $\vf$ is that, choosing $\lambda=\lambda_0(S,M)$ and $\vf$ accordingly,
we obtain that $h_\vf(M)=0$.

The classical Cheeger constants in \eqref{checon} correspond to
the case $S=\Delta$, $\lambda=0$, and $\vf=1$.

\subsection{Amenable actions and coverings}
\label{secamen}
Consider a right action of a countable group $\Gamma$ on a countable set $X$.
The action is called \textit{amenable} if there exists an invariant mean on $\ell^{\infty}(X)$;
that is, a linear map $\mu \colon \ell^{\infty}(X) \to \mathbb{R}$ such that
\begin{align*}
	\inf f \leq \mu (f) \leq \sup f \text{ and } \mu (g^{*} f) = \mu(f)
\end{align*}
for any $f \in \ell^{\infty}(X)$ and any $g \in \Gamma$.
The group $\Gamma$ is called amenable if the right action of $\Gamma$ on itself is amenable.

Clearly, any (right) action of $\Gamma $ on any finite set is amenable.
Furthermore, an action of $\Gamma$ on a countable set $X$ is amenable
if its restriction to a non-empty invariant subset of $X$ is amenable.

Amenability refers to some kind of asymptotic smallness of $X$ with respect to the action of $\Gamma$.
This is made precise by \cref{Folner} and \cref{Folner2} below,
due to F\o{}lner in the case of groups \cite[Main Theorem and Remark]{Fo}
and then extended to actions by Rosenblatt \cite[Theorems 4.4 and 4.9]{Ro}.
To state their results,
it will be convenient to use the following notion of boundary,
\begin{align*}
  \partial_S E
  = \{x\in E \mid \text{$xg\notin E$ for some $g\in S$}\}
  = \cup_{g\in S} (E \setminus Eg^{-1}),
\end{align*}
where $S\subseteq\Gamma$ and $E\subseteq X$.

\begin{thm}\label{Folner}
The following are equivalent:
\begin{enumerate}
\item\label{amen1}
The action of $\Gamma$ on $X$ is amenable.
\item\label{amen2}
For any finite $S\subseteq\Gamma$ and $\varepsilon > 0$,
there exists a finite $E \subseteq X$ such that $|Eg \smallsetminus E| < \ve|E|$ for any $g \in G$.\item\label{amen3}
For any finite $S\subseteq\Gamma$,
there exists a sequence of finite non-empty $E_n\subseteq X$ such that $|\partial_SE_n|/|E_n| \to 0$.
\item\label{amen4}
For any finite $S\subseteq\Gamma$,
there exist sequences of orbits $X_n\subseteq X$ of $\Gamma$
and of finite non-empty $E_n\subseteq X_n$ such that $|\partial_SE_n|/|E_n| \to 0$.
\end{enumerate}
\end{thm}

Although amenability is a global property,
\cref{Folner} characterizes it in terms of the action of finitely generated subgroups of $\Gamma$.

\begin{cor}\label{Folner2}
Assume that $\Gamma$ is finitely generated and that $S\subseteq\Gamma$ is a finite generating set.
Then the following are equivalent:
\begin{enumerate}
\item\label{amen1f}
The action of $\Gamma$ on $X$ is amenable.
\item\label{amen2f}
There is a sequence of finite non-empty subsets $E_n$ of $X$ such that $|\partial_SE_n|/|E_n|\to 0$.
\item\label{amen3f}
There are sequences of orbits $X_n\subseteq X$ of $\Gamma$
and of finite non-empty subsets $E_n$ of $X_n$ such that $|\partial_SE_n|/|E_n| \to 0$.
\end{enumerate}
\end{cor}

Let $p \colon M_{1} \to M_{0}$ be a covering, where $M_0$ is connected, but $M_1$ possibly not.
Fix $x \in M_{0}$, and consider the fundamental group $\pi_{1}(M_{0},x)$ with base point $x$.
For $g \in \pi_{1}(M_{0},x)$, let $c$ be a representative loop in $M_0$.
Given $y\in p^{-1}(x)$, lift $c$ to the path $c_y$ in $M_1$ starting at $y$,
and denote by $y \cdot g$ the endpoint of $c_y$.
In this way, we obtain a right action of $\pi_{1}(M_{0},x)$ on $p^{-1}(x)$,
which is called the \emph{monodromy action} of $p$.
Monodromy actions of $p$ for different choices of base point $x\in M_0$ are conjugate in the natural way.
The covering $p$ is called \emph{amenable} if one, and then any, of its monodromy actions is amenable.

\begin{exa}\label{exancc}
For any covering $p \colon M_{1} \to M_{0}$ (where $M_0$ is connected),
\begin{align*}
	p \sqcup \id \colon M_{1} \sqcup M_{0} \to M_{0}
\end{align*}
is an amenable covering.
\end{exa}

Recall that F\o{}lner's condition allows us to characterize amenability of an action of a group $\Gamma$ in terms of the restriction of the action to finitely generated subgroups of $\Gamma$.
In the context of coverings,
this is reflected by the following characterization of amenability. 

\begin{prop}\label{amecom}
Let $p \colon M_{1} \to M_{0}$ be a covering
and $K_1\subseteq K_2\subseteq \cdots$ be an exhaustion of $M_0$ by compact domains with smooth boundary.
Then $p$ is amenable if and only if the restrictions $p\colon p^{-1}(K_n)\to K_n$ of $p$ are amenable.
\end{prop}

\begin{proof}
Fix a base point $x \in M_{0}$ and consider a compact domain $K$ in $M_0$ with smooth boundary
containing $x$ in its interior.
It is immediate to verify that the monodromy action of $p\colon p^{-1}(K) \to K$
coincides with the monodromy action of $i_{*}\pi_{1}(K)$ on $p^{-1}(x)$,
where $i \colon K \to M_{0}$ stands for the inclusion.
It follows now from \cref{Folner} that if $p \colon M_{1} \to M_{1}$ is amenable,
then any restriction $p \colon p^{-1}(K) \to K$ is amenable.

Conversely, let $S$ be a finite subset of $\pi_{1}(M_{0})$ and, for any $g \in S$,
consider a representative loop $c_{g}$.
Since $S$ is finite,
there exists $n \in \mathbb{N}$ such that the images of all these loops are contained in $K_{n}$.
Since the restriction $p \colon p^{-1}(K_{n}) \to K_{n}$ is amenable,
we obtain from \cref{Folner} that there exist finite subsets $E_{n}$ of $p^{-1}(x)$
with $|\partial_{S}E_{n}|/|E_{n}| \rightarrow 0$.
We conclude from \cref{Folner} that $p$ is amenable, $S$ being arbitrary.
\end{proof}

\cref{amecom} illustrates the importance of considering non-connected covering spaces.
Namely, the preimage $p^{-1}(K)$ of a compact domain $K$ in $M_0$ with smooth boundary
may not be connected even if $M_1$ is.

\section{Monotonicity of $\lambda_0$}
\label{secele}

Different versions of the following result can be found in the literature;
see e.g. \cite[Theorem 7 and Corollary 1 of Theorem 4]{CY},
\cite[Theorem 1 and Corollary 1]{FS},
\cite[Theorem 1.2]{MP}, or \cite[Theorem 2.1]{Su}.

\begin{thm}\label{poslam}
Let $S$ be a Schr\"odinger operator on a Riemannian manifold $M$.
Then $\lambda_0(S,M)$ is the maximal $\lambda\in\R$
such that the equation $Sf=\lambda f$ has a positive solution.
\end{thm}

Returning to our standard setup of a Riemannian covering $p \colon M_{1} \to M_{0}$
with compatible Schr\"odinger operators $S_1$ and $S_0$,
we obtain a

\begin{proof}[Proof of \cref{monot}]
If $f_0$ is a positive function on $M_0$ solving $S_0f_0=\lambda f_0$,
then the positive function $f=f_0\circ p$ on $M_1$ solves $S_1f=\lambda f$.
\end{proof}

In \cite[Theorem 1.3]{BMP1},
we obtain \cref{monot} by an elementary argument which does not rely on \cref{poslam}.

\begin{proof}[Another proof of \cref{monot}]
Given any compactly supported Lipschitz function $f$ on $M_1$, its \emph{pushdown}
\begin{align}\label{pushd}
  f_0(x) = \left( \sum\nolimits_{y\in p^{-1}(x)}|f(y)|^2 \right)^{1/2}
\end{align}
is a Lipschitz function on $M_0$.
A straightforward calculation shows that the Rayleigh quotients satisfy $R_{S_{0}}(f_0)\le R_{S_{1}}(f)$.
Now \cref{monot} follows from the characterization of the bottom of the spectrum of Schr\"odinger operators
by Rayleigh quotients.
\end{proof}

\begin{rem}
It is easy to construct examples of Riemannian coverings $p\colon M_1\to M_0$,
where $\lambda_{\ess}(M_1)<\lambda_{\ess}(M_0)$,
contrary to the monotonicity of the bottom of the spectrum.
\end{rem}

\section{Amenability implies equality!}
\label{secame}

In this section, we review the history of \cref{tame}.
We start with Theorem 1 of Brooks in \cite{Br2} which extends, with almost identical proof,
the corresponding (half) of his \cite[Theorem 1]{Br1}.

Somewhat informally, Brooks defines a manifold to be of \emph{finite topological type}
if it is topologically the union of finitely many simplices \cite{Br2}.
For example, a surface is of finite topological type if it is of finite type in the usual sense,
that is, if it is diffeomorphic to a closed surface with finitely many punctures.

\begin{thm}[Brooks, Theorem 1 in \cite{Br2}]\label{Brooks2}
Suppose that $p$ is normal and that $M_{0}$ is of finite topological type and complete.
If $\Gamma$ is amenable,
then $\lambda_{0}(M_{1}) = \lambda_{0}(M_{0})$.
\end{thm}

The point of assuming topological finiteness of $M_0$ is that the covering $p$
admits a fundamental domain $F$ with finitely many sides.
The set
\begin{align}\label{genset}
	S = \{ s \in \Gamma \mid \text{$F$ and $sF$ meet along a codimension one face}\}
\end{align}
is then a finite and symmetric generating set of $\Gamma$.

\begin{proof}[Sketch of the proof of \cref{Brooks2} after \cite{Br2}]
In view of \cref{monot},
it is sufficient to construct compactly supported Lipschitz functions on $M_1$
with Rayleigh quotients at most $\lambda_0(M_0)+\ve$, for any $\ve>0$.
Now amenability of $\Gamma$ implies that there exists a sequence of finite subsets $E_n$ of $\Gamma$
such that $|\partial_SE_n|/|E_n|\to0$ as $n\to\infty$; see \cref{Folner2}.\ref{amen2f}.
Then setting $F_n=\cup_{g\in E_n}gF\subseteq M_1$,
Brooks uses the family of compactly supported functions
\begin{align*}
	\chi_n^\rho = \chi_n^\rho(x) =
	\begin{cases}
	1 &\text{if $x\in F_n$,} \\
	0 &\text{if $d(x,F_n)>\rho$,} \\
	1-d(x,F_n)/\rho &\text{otherwise,}
	\end{cases}
\end{align*}
to cut off (somewhat specific) lifts of functions from $C^\infty_c(M_0)$ to $M_1$.
He concludes the proof by estimating the Rayleigh quotients
of the resulting compactly supported Lipschitz functions on $M_1$ in terms of $|\partial_SE_n|/|E_n|$.  
\end{proof}

\begin{thm}[Ji, Li, \& Wang, Theorem 5.1 in \cite{JLW}]\label{JLWa}
Suppose that $p$ is normal and that $M_0$ is complete with Ricci curvature bounded from below.
Suppose furthermore that the volumes of geodesic balls in $M_1$ of radius one satisfy estimates
\begin{equation*}
  |B(x,1)| \ge C_\ve e^{-\ve r(x)}
\end{equation*}
for any $\ve>0$, where $r=r(x)$ denotes the distance to some origin in $M_1$.
If $\Gamma$ is amenable, then $\lambda_{0}(M_{1}) = \lambda_{0}(M_{0})$.
\end{thm}

\begin{proof}[About the proof of \cref{JLWa} by Ji, Li, and Wang]
They use estimates on the Green's function of $\Delta+(\lambda_0(M_1)-\ve)$
to construct a bounded positive function $u$ on $M_1$ which solves $\Delta u\ge(\lambda_0(M_1)-\ve)u$.
Then they use a mean for $\Gamma$ to push $u$ down to a function $v$ on $M_0$
which solves $\Delta v\ge(\lambda_0(M_1)-\ve)v$.
\end{proof}

\begin{thm}[B\'erard \& Castillon, Theorem 1.1 in \cite{BC}]\label{BeCa}
If $M_0$ is complete and $p$ is amenable, then $\lambda_0(S_1,M_1)=\lambda_0(S_0,M_0)$.
\end{thm}

Note that B\'erard and Castillon assume implicitly also that the group of covering transformations
resp.\ the fundamental group of the base is finitely generated
\cite[Sections 3.1 and 3.2]{BC}.

\begin{proof}[About the proof of \cref{BeCa} by B\'erard and Castillon]
They consider the case of normal coverings first.
In that case, their arguments are close to the ones of Brooks.
In a second step, they explain how to adapt the arguments to the case of general coverings.
\end{proof}

Adopting cut-offs of lifts of functions more carefully to the different competitors for $\lambda_0(S_{0},M_{0})$ separately
is the main new point in the proof of \cref{tame} in \cite{BMP1}.

\begin{proof}[Sketch of the proof of \cref{tame} after \cite{BMP1}]
Given a non-zero $f_{0} \in C^{\infty}_{c}(M_{0})$ and $\varepsilon > 0$,
we want to construct a non-zero $f\in \text{Lip}_{c}(M_{1})$
with $R_{S_{1}}(f) \leq R_{S_{0}}(f_{0}) + \varepsilon$. 
	
Fix $x \in M_{0}$ and $r > 0$ such that $\supp f_{0} \subseteq B(x,r)$,
where we measure distances with respect to a complete background metric.
A main step of the proof is the construction of a Lipschitz partition of unity on $M_{1}$,
consisting of functions $\varphi_{1}$ and $\varphi_{y}$, with $y \in p^{-1}(x)$,
such that $\varphi_{1} = 0$ in $p^{-1}(B(x,r))$, $\supp \varphi_{y} \subseteq B(y,r+1)$
and such that the Lipschitz constants of $\varphi_{y}$ with $y \in p^{-1}(x)$ do not depend on $y$.
For a finite subset $E$ of $p^{-1}(x)$ consider the Lipschitz function
\begin{align*}
	\chi_{E} := \sum\nolimits_{y \in E} \varphi_{y}.
\end{align*}
It follows from \cref{Folner} that there exists a finite set $E \subseteq p^{-1}(x)$
such that the function $f:=\chi_{E}f_1$ satisfies the desired inequality,
where $f_1$ is the lift of $f_{0}$ to $M_{1}$.
\end{proof}

\section{Equality implies amenability?}
\label{secname}

In this section,
we consider the problem of finding conditions under which non-amenability of the covering $p\colon M_1\to M_0$ implies the strict inequality $\lambda_{0}(M_1)>\lambda_0(M_0)$.
In other words, we search for conditions under which the equality $\lambda_{0}(M_1)=\lambda_0(M_0)$
implies amenability of $p$.
In a first part, we survey the development which lead to \cref{name}
up to and including the proof of \cref{name} by the second author,
in a second part we present a new argument
which leads to a simplification of the proof of \cref{name}.
In view of the previous section,
the corresponding statements below will actually assert equivalence to amenability of the covering.
We hope that our outlines of the proofs of the various results and ideas
turn out to be useful in further research on the subject.

\subsection{On the development towards \cref{name}}
\label{susname}
Our review starts with the work of Brooks in the case of the universal covering
of a closed Riemannian manifold $M_0$.
Then $\lambda_0(M_0)=0$.

\begin{thm}[Brooks, Theorem 1 in \cite{Br1}]\label{Brooks1}
	Suppose that $p\colon M_1\to M_0$ is the universal covering and that $M_{0}$ is closed.
	Then $\lambda_{0}(M_{1})=0$ if and only if the fundamental group of $M_0$ is amenable.
\end{thm}

\begin{proof}[Sketch of the proof of \cref{Brooks1} after \cite{Br1}]
Assume that $\lambda_0(M_1)=0$.

Let $F$ be a finite sided fundamental domain for the covering $p$,
and consider the finite generating set $S$ of $\Gamma$ as in \eqref{genset}.
Using \cref{Folner2}, it suffices to show that, for any $\varepsilon > 0$,
there exists a finite union $H$ of translates of $F$ such that
	\begin{equation}\label{small ratio 1}
	|\partial H| < \ve |H|.
	\end{equation}
To that end, we note first that the assumption $\lambda_{0}(M_{1}) = 0$
together with the Cheeger inequality implies that the Cheeger constant $h(M_{1}) = 0$;
that is, for any $n\in\N$, there exists a smoothly bounded, compact domain $D_{n}$ in $M_{1}$
such that $|\partial D_n|/|D_n| < 1/n$.
Next, we cover each $D_{n}$ with a finite union $K_{n}$ of translates of $F$.
Now $K_{n}$ does not need to satisfy \eqref{small ratio 1},
essentially due to the fact that the mean curvature of $\partial D_{n}$ need not be uniformly bounded.
This is the point which makes the proof technically involved and non-trivial.

	Taking into account that the action of $\Gamma$ on $M_{1}$ is cocompact,
it is not hard to see that there exist compact domains $W_{n}$ with smooth boundary
such that $K_n\subseteq W_n$ and such that the mean curvature of $\partial W_{n}$ is uniformly bounded.
Relying heavily on geometric measure theory,
Brooks then shows the existence of a minimizer of the Cheeger constant $h(\mathring{W}_n)$ in each $W_{n}$.
To be more precise, he shows that there exist domains $U_{n} \subseteq W_{n}$ with rectifiable $\partial U_{n}$
(roughly speaking, this means sufficiently regular to define volume and mean curvature)
which might touch $\partial W_n$,
such that $|\partial U_n|/|U_n| = h(\mathring{W}_{n})$ and such that the mean curvature of $\partial U_{n}$ is bounded in terms of $h(\mathring{W}_{n})$ and the mean curvature of $\partial W_{n}$.
In particular, we have that $|\partial U_n|/|U_n| < 1/n$
and that the mean curvature of $\partial U_{n}$ is uniformly bounded.
Covering $U_{n}$ with finite unions $H_{n}$ of translates of $F$
then gives rise to domains satisfying \eqref{small ratio 1}.	
\end{proof}

In \cite{Br2}, Brooks studies the noncompact case.
More precisely, not excluding compactness of $M_0$,
he assumes throughout \cite{Br2} that $M_0$ is complete
and of finite topological type as defined in the previous section.
He also assumes that the covering $p$ is normal with group $\Gamma$ of covering transformations.

In the proof, Brooks renormalizes the Laplace operator on $M_1$,
using the lift $\vf$ to $M_1$ of a positive $\lambda_0(M_0)$-harmonic function $\vf_0$ on $M_0$;
compare with \cref{suseren}.
He then adapts the Cheeger constant according to his needs. 
Namely, for any compact $L\subseteq F$, he sets
\begin{align*}
	h_\vf(F,L) = \inf \frac{|\partial D\cap\mathring{F}|_{\varphi}}{|D\cap\mathring{F}|_{\varphi}},
\end{align*}
where $D$ runs over the family of relatively compact domains $D$ in $F$ with smooth boundary $\partial D$
which intersects the boundary simplices of $F$ transversally, but such that $\bar D$ does not intersect $L$.

\begin{thm}[Brooks, Theorem 2 in \cite{Br2}]\label{Brooks3}
	Suppose that $p$ is normal and that $M_{0}$ is of finite topological type and complete.
	Suppose further that $h_\phi(F,L)>0$ for some compact $L\subseteq F$.
	Then $\lambda_{0}(M_{1})=\lambda_{0}(M_{0})$ if and only if $\Gamma$ is amenable.
\end{thm}

\begin{rem}[on the condition $h_\phi(F,L)>0$]\label{rembro1}
	Recall that the lower bound of the essential spectrum of $M_0$ is given by
	\begin{align*}
		\lambda_{\ess}(M_0) = \sup \lambda_0(M_0\setminus K), 
	\end{align*}
	where $K$ runs over the family of compact subsets of $M_0$.
	By \eqref{modche2}, 
	\begin{align*}
		\frac14h_{\vf_{0}}(M_0\setminus K)^2
		\le \lambda_0(M_0\setminus K) - \lambda_0(M_0).
	\end{align*}
	For a non-empty compact $K\subseteq M_0$, let $L=p^{-1}(K)\cap F$.
	Then
	\begin{align*}
		h_{\vf_{0}}(M_0\setminus K) \ge h_\vf(F,L).
	\end{align*}
	We conclude that $\lambda_{\ess}(M_0)>\lambda_0(M_0)$
	if $h_\vf(F,L)>0$ for some compact $L\subseteq F$.
	This observation let Brooks to the question whether the condition
	$\lambda_{\ess}(M_0)>\lambda_0(M_0)$
	is a valid replacement of his condition $h_\vf(F,L)>0$.
	(Together with \cref{tame}, this is achieved by \cref{name}.)
\end{rem}

\begin{proof}[Sketch of the proof of \cref{Brooks3} after \cite{Br2}]
Assume that
\begin{align*}
	\lambda_{0}(M_{1}) = \lambda_{0}(M_{0}).
\end{align*}
	Consider a positive $\psi_{0} \in C^{\infty}(M_{0})$ with $\psi_{0} = \varphi_{0}$
	outside a compact neighborhood of $K := p(L)$, $\psi_{0} = 1$ in a neighborhood of $K$,
	and denote by $\psi$ its lift to $M_{1}$.
	Similarly to the proof of \cref{Brooks1}, given a minimizing sequence $(D_{n})_{n \in \mathbb{N}}$
	for $h_{\psi}(M_{1})$, Brooks asserts that $D_{n}$ is contained
	in a smoothly bounded, compact domain $W_{n}$ satisfying certain requirements,
	in particular that $\partial W_{n}$ has uniformly bounded mean curvature in $p^{-1}(K)$.
	Again relying on geometric measure theory,
	he asserts that there exists a minimizer $U_{n}$ of $h_{\psi}(\mathring{W}_{n})$ in each $W_{n}$
	and that $\partial U_{n}$ has uniformly bounded mean curvature in $p^{-1}(K)$.
	He then estimates the isoperimetric ratio of the $U_{n}$ in each translate of $F$ separately.
	From the assumption that the covering is non-amenable, Brooks concludes that $h_{\psi}(M_{1}) > 0$,
	which implies that $h_{\varphi}(M_{1}) > 0$.
	The proof is then completed by the modified Cheeger inequality from \eqref{modche2}.
\end{proof}

The preceding result of Brooks motivated Roblin and Tapie to establish another extension of \cref{Brooks1}.
As in \cref{Brooks3}, their extension also involves assumptions on a fundamental domain of the covering. However, it is worth to mention that their proof avoids using geometric measure theory.

In order to state their result, we need some definitions.
Suppose that $p$ is normal.
Consider a positive $\lambda_{0}(M_{0})$-harmonic function $\varphi_{0}$ on $M_{0}$.
Roblin and Tapie call a fundamental domain $F$ of $p$ \emph{spectrally optimal}
(with respect to $\vf_0$) if $\partial F$ is piecewise $C^{1}$
and the lift $\varphi$ of $\varphi_{0}$ to $F$ satisfies Neumann boundary conditions along $\partial F$.
By the square-integrability and positivity of $\vf_0$, if $F$ is a spectrally optimal fundamental domain,
then the bottom of its Neumann spectrum is given by $\lambda_{0}^{N}(F) = \lambda_{0}(M_{0})$
\cite[Lemme 4.2]{RT}.

\begin{thm}[Roblin \& Tapie, Theorem 4.3 in \cite{RT}]\label{RoTa}
	Suppose that $p$ is normal, that $M_0$ is complete, 
	and that $\lambda_0(M_0)$ is an eigenvalue of $M_0$.
	Let $\vf_0$ be a positive $\lambda_0(M_0)$-eigenfunction on $M_0$.
	Suppose further that $p$ has a spectrally optimal fundamental domain $F$
	such that $\lambda_{0}(M)$ is an isolated eigenvalue of multiplicity one
	of the Neumann spectrum of $F$.
	Then $\lambda_{0}(M_{1})=\lambda_{0}(M_{0})$ if and only if $\Gamma$ is amenable.
\end{thm}

Roblin and Tapie assume implicitly that $\Gamma$ is finitely generated.
More precisely, they assume that the symmetric generating set $S$ of $\Gamma$
as in \eqref{genset} is finite \cite[Page 72]{RT}.

\begin{rem}[on the conditions on $F$ and $\vf$]\label{rembro2}
As we already pointed out above,
since the lift $\vf$ of $\vf_{0}$ to $F$ satisfies Neumann boundary conditions,
it is actually a Neumann-eigenfunction corresponding to $\lambda_{0}^{N}(F)$,
and hence $\lambda_{0}^{N}(F)=\lambda_0(M_0)$.
Moreover, the assumption that $\lambda_{0}^{N}(F)$
is an isolated point of the Neumann spectrum of $F$ of multiplicity one implies that
\begin{align*}
  \lambda_1^N(F) := \inf \la \Delta f,f\ra_{L^2(F)}/\|f\|_{L^2(F)}^2
  > \lambda_0^N(F) = \lambda_0(M_0),
\end{align*}
where the infimum is taken over all smooth functions $f$ on $F=\bar F$ with compact support
which are $L^2$-perpendicular to $\vf$.
Now the lift to $F$ of any smooth function on $M_0$ with compact support,
which is $L^2$-perpendicular to $\vf_0$, is of this kind,
and hence
\begin{align*}
  \lambda_1(M_0) := \inf \la \Delta f,f\ra_{L^2(M_0)}/\|f\|_{L^2(M_0)}^2
  \ge \lambda_1^N(F) > \lambda_0(M_0),
\end{align*}
where the infimum is now taken over all smooth functions $f$ on $M_0$ with compact support
which are $L^2$-perpendicular to $\vf_0$.
It should be noticed that $\lambda_1(M_0)=\inf (\sigma(M_0)\setminus\{\lambda_0(M_0)\})$.
In particular, the assumptions of \cref{RoTa} imply that
$\lambda_{\ess}(M_{0})\ge \lambda_1(M_0) > \lambda_{0}(M_{0})$.
\end{rem}

\begin{proof}[Sketch of the proof of \cref{RoTa} after \cite{RT}]
Assume that
\begin{align*}
	\lambda_{0}(M_{1}) = \lambda_{0}(M_{0}).
\end{align*}
Let $\varphi_{0}$ be a positive $\lambda_{0}(M_{0})$-harmonic function on $M_{0}$ of $L^2$-norm one.
Then, on any translate of $F$,
the lift $\varphi$ of $\varphi_{0}$ to $M_{1}$ is square-integrable
and satisfies Neumann boundary conditions along its boundary.
Now for any $\varepsilon > 0$,
there exists $f \in C^{\infty}_{c}(M_{1})$ with $R(f) < \lambda_{0}(M_{1}) + \varepsilon$.
Given $g \in \Gamma$, consider the orthogonal projection of $f$ on $\varphi$ in $L^{2}(g F)$;
that is, write
\begin{align*}
  f = b(g) \varphi + h_{g} \text{ in } L^{2}(gF),
  \text{ where } b(g) := \int_{g F} f \varphi\dv.
\end{align*}
Since $h_{g}$ is perpendicular to $\varphi$ in $L^{2}(g F)$,
\begin{align*}
	\int_{g F} |\grad h_{g}|^{2} \geq \lambda_{1}^{N}(F) \int_{g F} h_{g}^{2},
\end{align*}
where $\lambda_{1}^{N}(F)$ is the infimum of the Neumann spectrum of $F$ with $\lambda_{0}^{N}(F)$ removed.
In this way, Roblin and Tapie obtain a function $b \colon \Gamma \to \R$, which turns out to satisfy
\begin{align*}
\sum_{g \in \Gamma} \sum_{s \in S} (b(sg) - b(g))^{2} < \delta(\varepsilon) \sum_{g \in \Gamma} b_{g}^{2},
\end{align*}
where $\delta(\varepsilon) \rightarrow 0$ as $\varepsilon \rightarrow 0$.
This yields that the bottom of the spectrum of the discrete Laplacian (with respect to $S$) on $\Gamma$ is zero. 
Then the discrete version of the Cheeger inequality (cf.\ for instance \cite[Proposition 4.6]{RT})
implies that the second condition of \cref{Folner2} is satisfied, and hence that $\Gamma$ is amenable.
\end{proof}

Recall that Brooks's proof of \cref{Brooks1} relies on geometric measure theory.
More precisely, Brooks uses geometric measure theory
in order to pass from a minimizing sequence of domains $D_{n}$ for the Cheeger constant
to a minimizing sequence of domains $U_{n}$ for which we have estimates on the volume of the collars
of radius $r$ about $\partial U_n$.

Buser encountered a very similar problem in the establishment of his converse to the Cheeger inequality
and resolved this issue relying only on the Bishop-Gromov volume comparison theorem \cite{Bu}.
To state a consequence of his considerations,
we use the notation $A^r$ for the $r$-neighborhood of a set $A$ in $M$.

\begin{lem}[Buser, consequence of Lemma 7.2 in \cite{Bu}]\label{Buser}
	Let $M$ be a possibly non-connected, complete Riemannian manifold
	with Ricci curvature boun\-ded from below.
	If $h(M) = 0$, then, for any $\ve,r>0$,
	there exists a bounded open subset $A \subseteq M$ such that $|A^{r}\setminus A| < \varepsilon |A|$.
\end{lem}

Buser's approach to this problem gave rise to a different extension of Brooks's result,
not involving any assumptions on any fundamental domains,
and valid also for non-normal coverings and Schr\"odinger operators
(with conditions on the potential).

\begin{thm}[Ballmann, Matthiesen, \& Polymerakis, Theorem 1.3 in \cite{BMP2}]\label{BMP2 result}
Suppose that the Ricci curvature of $M_0$ is bounded from below.
Let $S_0=\Delta+V$ be a Schr\"{o}dinger operator on $M_{0}$
with $V$ and $\grad V$ bounded, and let $S_{1}$ be its lift to $M_{1}$.
Assume that $\lambda_{\ess}(S_{0},M_0)>\lambda_{0}(S_{0},M_0)$.
Then $\lambda_{0}(S_{1},M_1)=\lambda_{0}(S_{0},M_0)$ if and only if $p$ is amenable.
\end{thm}

\begin{proof}[Sketch of the proof of \cref{BMP2 result} after \cite{BMP2}]
Assume that
\begin{align*}
	\lambda_{0}(S_{1},M_1) = \lambda_{0}(S_{0},M_0).
\end{align*}
Considering the renormalization $S_{\varphi}$ of $S_{1}-\lambda_0(S_0,M_0)$
with respect to the lift $\vf$ of a positive $\lambda_{0}(S_{0},M_{0})$-eigenfunction of $S_{0}$
(see \eqref{bottom2}),
we obtain that $\lambda_{0}(S_{\varphi},M_{1}) = 0$.
Then the modified Cheeger inequality yields that $h_{\varphi}(M_{1}) = 0$ (see \eqref{modche2}).
	
	Following the arguments of Buser, it follows that, for any $r > 0$ and $n \in \N$,
	there exists an open, bounded $A_{n} \subseteq M_{1}$
	such that $| A_{n}^{r} \smallsetminus A_{n} |_{\varphi} < |A_{n}|_{\varphi}/n$.
	For this, it is important that $\varphi$ satisfies uniform Harnack estimates,
	which, under our assumptions, follows from the Cheng-Yau gradient estimate
	(cf. \cite[Theorem 6]{CY}).
	Consider the compactly supported Lipschitz function 
	\begin{align*}
		f_n(y) = 
		\begin{cases}
			1 - d(y,A_{n})/r &\text{if $d(y,A_{n})< r$,} \\
			0  &\text{otherwise,}
		\end{cases}
	\end{align*}
	on $M_{1}$.
	It is straightforward to verify that the Raileigh quotients
	\begin{align*}
	  R_{S_{\varphi}}(f_{n}) = \frac{\la S_\vf f_n,f_n\ra_{L^2(M_1,\mu)}}{\|f_n\|_{L^2(M_1,\mu)}^2} \rightarrow 0
	\end{align*}
	or, equivalently,
	that $R_{S_{1}}(f_{n} \varphi) \rightarrow \lambda_{0}(S_{0},M_{0})$ as $n \rightarrow \infty$.
	
	The assumption that $\lambda_{0}(S_{0},M_{0}) < \lambda_{\ess}(S_{0},M_{0})$ implies
	that there exists a compact domain $K$ of $M_{0}$
	such that $\lambda_{0}(S_{0}, M_{0} \smallsetminus K) > \lambda_{0}(S_{0},M_{0})$.
	Since the restriction $p \colon M_{1} \smallsetminus p^{-1}(K) \to M_{0} \smallsetminus K$
	is a Riemannian covering of possibly non-connected manifolds, it follows that
	\begin{equation}\label{BMP2 inequality}
	\lambda_{0}(S_{1}, M_{1} \smallsetminus p^{-1}(K))
	\geq \lambda_{0}(S_{0}, M_{0} \smallsetminus K)
	> \lambda_{0}(S_{0},M_{0})
	= \lambda_{0}(S_{1},M_{1}).
	\end{equation}
	This shows that for any $\varepsilon > 0$, there exists $n \in \mathbb{N}$ and $x \in K$
	such that $f_n$ is differentiable in each $y\in p^{-1}(x)$ and such that
	\begin{align*}
	\sum_{y \in p^{-1}(x)} |\grad f_{n}(y)|^{2} < \varepsilon  \sum_{y \in p^{-1}(x)} f_{n}^{2}(y).
	\end{align*}
	Indeed, otherwise we would be able to cut off $f_{n}$ with a lifted function $\chi$
	and obtain that $\supp \chi f_{n} \subseteq M_{1} \smallsetminus p^{-1}(K)$
	and $\mathcal{R}_{S_{1}}(\chi f_{n}) \rightarrow \lambda_{0}(S_{1},M_{1})$
	in contradiction with \eqref{BMP2 inequality}.
	From the definition of $f_{n}$, it is easy to see that the above estimate gives
	that $|p^{-1}(x) \cap (A^{r} \smallsetminus A)| < \varepsilon r^{2} |p^{-1}(x) \cap A|$.
	This, together with \cref{Folner} and the fact that $K$ is bounded,
	shows that the covering is amenable. 
\end{proof}

Based on \cref{BMP2 result},
the second author was able to establish \cref{name} in full generality \cite{Po2, Po3}.
An important step of the proof is the establishment of an analogue
of Brooks's \cref{Brooks1}	for manifolds with boundary,
where we are interested in the Neumann spectrum of (the Laplacian of) the manifold.
To that end, recall that the \emph{bottom of the Neumann spectrum}
of a Riemannian manifold $M$ with boundary is given by
\begin{align}\label{bottom Neumann}
	\lambda_{0}^{N}(M) = \inf R(f)
\end{align}
with $R(f)$ as in \eqref{raylei} (and $V=0$),
where the infimum is taken over all non-zero $f \in C^{\infty}_{c}(M)$.
It should be noticed that the test functions do not have to satisfy any boundary condition.

\begin{thm}[Polymerakis, Theorem 4.1 of \cite{Po3}]\label{neuma}
	Let $p \colon M_{1} \to M_{0}$ be a Riemannian covering with $M_{0}$ compact and $M_{1}$ possibly non-connected. If $\lambda_{0}^{N}(M_{1}) = 0$, then the covering is amenable.
\end{thm}

\begin{proof}[Sketch of proof of \cref{neuma} after \cite{Po3}]
A first observation is that the proof of \cref{BMP2 result}
also works in the case where $M_1$ is not connected.
Now by gluing cylinders along the boun\-daries of $M_0$ and $M_1$,
one obtains complete manifolds with bounds on their Ricci curvature.
Then \cref{neuma} follows from \cref{BMP2 result},
using appropriately chosen Schr\"{o}dinger operators on the new manifolds.
\end{proof}
	
\begin{proof}[Sketch of proof of \cref{name} after \cite{Po3}]
	Assume that $p$ is non-amenable and,
	to arrive at a contradiction, that
\begin{align*}
	\lambda_{0}(S_1,M_1) = \lambda_0(S_0,M_0).
\end{align*}
	According to \eqref{bottspec},
	there exists a sequence $(f_{n})_{n \in \mathbb{N}}$ in $C^{\infty}_{c}(M_{1})$
	with $L^2$-norm one and $R_{S_{1}}(f_{n}) \rightarrow \lambda_{0}(S_{1},M_1)$.
	Since the covering is non-amenable,
	we obtain from \cref{amecom} that there exists a smoothly bounded, compact domain $K \subseteq M_{0}$ 
	such that the covering $p \colon p^{-1}(K) \to K$ is non-amenable.
	Then \cref{neuma} yields that the bottom of the Neumann spectrum of the Laplacian
	satisfies $\lambda_{0}^{N}(p^{-1}(K)) > 0$.
	This allows us to cut off the functions $f_{n}$ and obtain a sequence with the same properties
	(also denoted by $(f_{n})$) such that $\supp f_{n} \cap p^{-1}(K) = \emptyset$.
	
	It is easy to see that the sequence $(g_{n})_{n \in\N}$ in $\text{Lip}_{c}(M_{0})$ 
	consisting of the pushdowns of $f_{n}$ (defined in \eqref{pushd}) satisfies $\| g_{n} \|_{L^{2}} = 1$, $\supp g_{n} \cap K = \emptyset$, 
	and $R_{S_{0}}(g_{n}) \rightarrow \lambda_{0}(S_{0},M_{0})$.
	The assumption that $\lambda_{0}(S_{0},M_{0}) < \lambda_{\ess}(S_{0},M_{0})$ yields that,
	after passing to a subsequence if necessary,
	we have $g_{n} \rightarrow \varphi$ in $L^{2}(M_{0})$ for some positive $\lambda_{0}(S_{0},M_{0})$-harmonic (with respect to $S_{0}$) function $\varphi \in C^{\infty}(M_{0})$.	This is a contradiction, since such a $\varphi$ is positive,
	whereas $\supp g_{n} \cap K = \emptyset$.
\end{proof}

\subsection{Simplification of the proof of \cref{name}}
\label{susim}
We now discuss a simplified proof for this theorem.
The point of this alternative proof is a simpler way of getting \cref{neuma},
not using \cref{BMP2 result}, but relying only on an extension of \cref{Brooks1}.

Let $p \colon M_{1} \to M_{0}$ be a Riemannian covering of complete manifolds and consider $x \in M_{1}$. For $y \in p^{-1}(x)$, the \emph{Dirichlet domain} of $p$ centered at $y$ is defined to be
\begin{align*}
	D_{y} 
	= \{ z \in M_{1} \mid d(z,y) \leq d(z,y^{\prime}) \text{ for any } y^{\prime} \in p^{-1}(x) \}.
\end{align*}
For $r > 0$ we denote by $G_{r}$ the finite subset of $g\in\pi_{1}(M_{0},x)$
such that $g$ contains a loop of length at most $r$.

\begin{thm}[Polymerakis, Theorem 6.1 in \cite{Po}]\label{closed base}
Let $p \colon M_{1} \to M_{0}$ be a Riemannian covering with $M_{0}$ closed and $M_{1}$ possibly non-connected.
Then $\lambda_{0}(M_{1}) = 0$ implies that $p$ is amenable.
\end{thm}

Strictly speaking, allowing that $M_1$ need not be connected is not contained in \cite{Po}.
It is, however, only a trivial extension of \cite[Theorem 6.1]{Po}
and will be used in our argument below.

\begin{proof}[Sketch of the proof of \cref{closed base} after \cite{Po}]
Assume that $\lambda_{0}(M_{1}) = 0$.
	Then the Cheeger inequality implies that $h(M_{1}) = 0$.
	In virtue of \cref{Buser},
	it follows that, for any $\ve> 0$ and $r> 2\diam M_0$,
	there exists a bounded $A \subseteq M_1$ such that
	\begin{equation*}
	| A^{3r} \smallsetminus A | < \varepsilon |A|.
	\end{equation*} 
	Consider the finite set $F := p^{-1}(x) \cap A^{r}$.
	From the fact that
	\begin{align*}
	\diam D_{y} \le 2\diam M_{0} < r
	\end{align*}
we readily see that $A$ is contained in $\cup_{y\in F}D_{y}$.
It is not hard to see that for $g \in G_{r}$ and $y \in Fg \setminus F$,
we have that $D_{y}$ is contained in $A^{3r} \setminus A$.
Using that $|D_{y}| = |M_{0}|$
and that the intersection of different $D_{y}$'s is of measure zero,
we deduce that
 \[
	|Fg \smallsetminus F| \leq |M_{0}|^{-1} |A^{3r} \smallsetminus A| < \varepsilon |M_{0}|^{-1}|A| \leq \varepsilon |F|
\]
for any $g \in G_{r}$.
We conclude from \cref{Folner} that the covering is amenable, since $\varepsilon > 0$ is arbitrary
and any finite $G \subset \pi_{1}(M_{0})$ is contained in $G_{r}$ for some $r > 2 \text{diam}(M_{0})$.
\end{proof}

\begin{proof}[Another proof of \cref{neuma}]
Change the given Riemannian metric of $M_0$ in a neighborhood $U\cong\partial M_0\times[0,\ve)$
of $\partial M_0$ so that the new metric is a product metric $g_0\times dr^2$ on $U$
and endow $M_1$ with the lifted metric.
Since $M_0$ is compact, the old and new Riemannian metrics on $M_0$ and $M_1$ are uniformly equivalent,
and hence also $\lambda_0^N(M_1)=0$ with respect to the new metric.

Denote by $2M_0=M_0\hat{\cup}M_0$ the manifold obtained by gluing two copies of $M_0$
along their common boundary, and define $2M_1$ correspondingly.
Since the new metrics from above are product metrics in neighborhoods of the boundaries,
they fit together to define (smooth) Riemannian metrics on  $2M_0$ and $2M_1$
so that $p$ extends to a Riemannian covering $2p\colon2M_1\to2M_0$.
Since $\lambda_0^N(M_1)=0$ with respect to the new metric
and test functions in $C^{\infty}_c(M_1)$ can be doubled to test functions in ${\rm Lip}_c(2M_1)$
with the same Rayleigh quotient, we get that $\lambda_0(2M_1)=0$.
Since $2M_0$ is closed, we conclude from \cref{closed base} that the covering $2p$ is amenable.
By \cref{amecom}, $p$ is then also amenable.
\end{proof}

\section{The case of hyperbolic manifolds}
\label{sechype}

We say that a Riemannian manifold $M$ of dimension $m$ is \emph{real hyperbolic}
if it is a quotient of real hyperbolic space $H_\R^m$ by a discrete group $\Gamma$  of isometries of $H_\R^m$.
Then the \emph{critical exponent} $\delta(M)$ is the infimum of the set of $s\in\R$ such that the Poincar\'e series
\begin{align*}
  g(x,y,s) = \sum\nolimits_{\gamma\in\Gamma} e^{-sd(x,\gamma(y))}
\end{align*}
converges for all $x,y\in H_\R^m$.
By \cite[Theorem 2.17]{Su}, we have
\begin{align*}
  \lambda_0(M) =
  \begin{cases}
  \delta(M)(m-1-\delta(M)) &\text{if $\delta(M)\ge(m-1)/2$}, \\
  \lambda_0(H_\R^m) = (m-1)^2/4 &\text{if $\delta(M)\le(m-1)/2$}.
  \end{cases}
\end{align*}
In particular, $\lambda_0(M)<\lambda_0(H_\R^m)$ if and only if $\delta(M)>(m-1)/2$.

We say that $M=\Gamma\backslash H_\R^m$ is \emph{geometrically finite (in the classical sense)}
if the covering $H_\R^m\to M$ admits a fundamental domain
which is bounded by finitely many hyperbolic hyperplanes.
In this case, the critical exponent $\delta(\Gamma)$ agrees with the Hausdorff dimension
of the limit set of $\Gamma$; compare with \cite[Theorem 2.21]{Su}.

Closed real hyperbolic manifolds are geometrically finite with $\lambda_{\ess}(M)=\infty$.
Since ends of non-closed real hyperbolic manifolds of finite volume are cusps,
they are geometrically finite with $\lambda_{\ess}(M)=(m-1)^2/4$.
Finally, by the work of Lax and Phillips,
$\lambda_{\ess}(M)=(m-1)^2/4$ if $M$ is geometrically finite of infinite volume \cite[p.\;281]{LP}.
In conclusion, if $M$ is geometrically finite in the classical sense,
then the assumptions of \cref{name} are satisfied if and only if $\delta(M)>(m-1)/2$.

Consider now any of the hyperbolic spaces $H_\R^m$, $H_{\C}^n$, $H_{\H}^n$, or $H_{\O}^2$.
Let $H$ be one of them, and let $m$ be the dimension of $H$.
Normalize the metric of $H$ so that the maximum of the sectional curvature of $H$ is $-1$.
Then the volume entropy of $H$, that is,
the exponential growth rate of the volume of metric balls in $H$, is given by
\begin{align*}
  h(H) =
\begin{cases}
  m-1 &\text{if $H=H_\R^m$}, \\
  m &\text{if $m=2n$ and $H=H_\C^n$}, \\
  m+2 &\text{if $m=4n$ and $H=H_\H^n$}, \\
  22 &\text{if $m=16$ and $H=H_\O^2$}. \\
\end{cases}
\end{align*}
Recall that $\lambda_0(H)=h(H)^2/4$.

Consider a quotient $M=\Gamma\backslash H$, also called  a \emph{hyperbolic manifold},
where $\Gamma$ is a discrete group of isometries of $H$.
Let $\Omega=S\setminus\Lambda$, where $S$ denotes the sphere at infinity of $H$
and $\Lambda\subseteq S$ the limit set of $\Gamma$.
Following Bowditch \cite{Bo2},
we say that $M$ is \emph{geometrically finite}
if $\Gamma\backslash(H\cup\Omega)$ has finitely many ends
and each of them is a cusp.
In the case of real hyperbolic manifolds, geometric finiteness in the classical sense as defined above 
implies geometric finiteness as defined here; see \cite[pages 289 and 302f]{Bo1}.
If $M$ is geometrically finite, then  $\lambda_{\ess}(M)=\lambda_0(H)$
\cite[Theorem]{Ha} or \cite[Theorem 1.1]{Li}.
Together with \cref{name}, we get the following consequence.

\begin{thm}\label{gefin}
Let $p\colon M_1\to M_0$ be a Riemannian covering of quotients of $H$, where $M_0$ is geometrically finite.
Then there are the following two cases:
\begin{enumerate}
\item\label{gefia}
If $\lambda_0(M_0)<\lambda_0(H)$, then $\lambda_0(M_1) = \lambda_0(M_0)$
if and only if $p$ is amenable.
\item\label{gefib}
If $\lambda_0(M_0)=\lambda_0(H)$, then $\lambda_0(M_1) = \lambda_0(M_0)$.
\end{enumerate}
\end{thm}

\cref{gefin} extends \cite[Theorem 1.6]{BMP2}, which dealt with real hyperbolic manifolds
which are geometrically finite in the classical sense.

\begin{rem}\label{remcoco}
Let $M=\Gamma\backslash H_\R^m$ be real hyperbolic manifold
which is geometrically finite in the classical sense.
We say that $M$ is \emph{convex cocompact} if it does not have cusps or, equivalently,
if $\Gamma$ does not contain parabolic isometries.
\cref{gefin}.\ref{gefia} for normal coverings $p\colon M_1\to M_0$
of real hyperbolic manifolds with $M_0$ convex cocompact is due to Brooks (\cite[Theorem 3]{Br2}).
The main point in his proof is to show that the covering $p$
admits a fundamental domain satisfying the isoperimetric inequality required in \cref{Brooks3}.
Another proof in this case was obtained by Roblin and Tapie
who show in \cite[Th\'eor\`eme 5.1]{RT} that the covering $p$
then admits a fundamental domain satisfying the requirements of \cref{RoTa}.
\end{rem}

\begin{rem}\label{remhype}
The case $\lambda_0(M_0)=\lambda_0(H)$ in \cref{gefin}.\ref{gefib}
yields examples of non-amenable Riemannian coverings
$p\colon M_1\to M_0$ of hyperbolic manifolds where the strict inequality $\lambda_0(M_1) > \lambda_0(M_0)$ fails.
A first example $M_0$, where this can happen, is described in \cite[Section 1]{Br2}.
Namely, let $\Gamma$ be the discrete group of motions of $H_\R^3$
generated by the reflections about three disjoint hyperbolic planes in $H_\R^3$,
which mutually touch at infinity.
Then the limit set of $\Gamma$, and any subgroup $\Gamma_0$ of finite index in $\Gamma$,
is the circle in the sphere at infinity,
which passes through the three points of tangency of the hyperbolic planes at infinity.
If such a subgroup $\Gamma_0$ is torsion-free,
then the quotient $M_0=\Gamma_0\backslash H_\R^3$ is a convex cocompact real hyperbolic manifold
with critical exponent equal to one (the Hausdorff dimension of the limit set),
and therefore is of the kind required in \cref{gefin}.\ref{gefib}.
Now $\Gamma$ is a finitely generated linear group,
and hence it has plenty of torsion free subgroups $\Gamma_0$ of finite index.

Finally, for any non-elementary hyperbolic manifold $M=\Gamma\backslash H$
and independently of $\lambda_0(M)$,
the normal Riemannian covering $H\to M$ is non-amenable
since $\Gamma$ contains the free subgroup in two generators as a subgroup. 
\end{rem}

\begin{exa}\label{exa00}
Let $S_n$, $n\ge1$, be a compact hyperbolic surface with two boundary geodesics,
which have length $1/n$ and $1/(n+1)$,
and glue them consecutively along the boundary geodesics of common length
to obtain a hyperbolic surface $S'$ of infinite type with one boundary geodesic of length one.
Attach a hyperbolic surface $S_0$ with one boundary circle of length one to $S'$
to obtain a  hyperbolic surfaces $S$ of infinite type.
The Lipschitz test functions $f_n$ on $S$
(varying the test functions used by Buser to get small eigenvalues)
such that $f_n=0$ outside the piece $S_n$,
$f_n=1$ in $S_n$ outside the collar $C_n\subseteq S_n$ of radius one about $\partial S_n$,
and $f_n(x)=d(x,\partial S_n)$ on $C_n$,
have Rayleigh quotients tending to zero as $n\to\infty$.
Hence we have $\lambda_{\ess}(S)=\lambda_0(S)=0$.
\end{exa}

\begin{exa}\label{exabcd}
We vary \cite[Example 4.1]{BCD} of Buser, Colbois, and Dodziuk.
Let $S_0$ be a closed hyperbolic surface
and $C$ be a union of $n\ge2$ disjoint simple closed geodesics
$c_1,\dots,c_n$ of respective lengths $\ell_1,\dots,\ell_n$
such that $S_{0}\setminus C$ is connected.
Cut $S_0$ along $C$ to obtain a compact and connected hyperbolic surface $F_0$
with $2n$ boundary geodesics $c_k^\pm$.

Let $T_n$ be the tree, all of whose vertices have valence $2n$.
Label the outgoing edges of each vertex by the $2n$ numbers $\pm k$, $1\le k\le n$,
such that the two labels of any edge of $T_n$ have the same absolute value,
but opposite signs.
Then the free group $\Gamma$ with $n$ generators $a_1,\dots,a_n$ acts on the vertices of $T_n$
such that $a_k^{\pm1}x=y$ if the outgoing edge from $x$ with label $\pm k$
is equal to the outgoing edge of $y$ with label $\mp k$.

For each vertex $x$ of $T_n$, let $F_x$ be a copy of $F_0$.
Glue back the boundary geodesic $c_k^\pm$ of $F_x$,
but now to the boundary geodesic $c_k^\mp$ of $F_y$, if $a_k^{\pm1}x=y$.
In this way, we obtain a hyperbolic surface $S$
on which $\Gamma$ acts isometrically as a group of covering transformations with quotient $S_0$.
Since $\Gamma$ is infinite, we conclude that $\lambda_{\ess}(S)=\lambda_0(S)$.
On the other hand, $\lambda_{\ess}(S_0)=\infty>\lambda_0(S_0)=0$.
Since $n\ge2$, $\Gamma$ is non-amenable,
hence \cite[Theorem 2]{Br2} (or \cref{name} above) implies that $\lambda_0(S)>0$.

By choosing the hyperbolic metric on $S_0$ appropriately,
we can make the lengths $\ell_1,\dots,\ell_n$ as small as we please,
and with them also $\lambda_0(S_0)$.
In particular, we can get $\lambda_0(S)<1/4=\lambda_0(H_\R^2)$
and then the bottom of the spectrum falls strictly
under the universal covering $H_\R^2\to S$,
although \cref{name} does not apply because $\lambda_{\ess}(S)=\lambda_0(S)$;
compare with \cref{question}.
The strict inequality $\lambda_0(M_1)<\lambda_0(H)$ actually holds
for any non-simply connected normal covering space $M_1$ of any closed hyperbolic manifold $M_0$; see \cite[Corollary 5]{Sa}.

In \cite[Example 4.1]{BCD},
Buser, Colbois, and Dodziuk also obtain examples of hyperbolic surfaces $S$ of infinite type
with arbitrarily small $\lambda_{\rm ess}(S)>0$
and infinitely many eigenvalues below $\lambda_{\rm ess}(S)$.
In particular, $\lambda_{\rm ess}(S)>\lambda_0(S)$
and $\lambda_0(S)>0$ (since $|S|=\infty$ and $\lambda_{\rm ess}(S)>0$).
Their construction and arguments extend to the above examples. 
\end{exa}

\begin{thm}\label{thmkk}
Let $p\colon M_1\to M_0$ be a Riemannian covering
of complete and connected Riemannian manifolds.
Assume that there is a geometrically finite hyperbolic manifold $M_0'=\Gamma\backslash H$ of infinite volume
such that $M_0\setminus K$ is isometric to $M_0'\setminus K'$
for some compact domains $K\subseteq M_0$ and $K'\subseteq M_0'$.
Then $\lambda_{\ess}(M_0)=\lambda_0(H)$,
and there are the following two cases:
\begin{enumerate}
\item\label{kka}
If $\lambda_0(M_0)<\lambda_0(H)$, then $\lambda_0(M_1) = \lambda_0(M_0)$
if and only if $p$ is amenable.
\item\label{kkb}
If $\lambda_0(M_0)=\lambda_0(H)$, then $\lambda_0(M_1) = \lambda_0(M_0)$.
\end{enumerate}
\end{thm}

If $\mathring{K}\ne\emptyset$,
then the condition  $\lambda_0(M_0)<\lambda_0(H)$ is easy to achieve
by modifying the metric on $K$ appropriately.

\begin{proof}
Since the essential spectrum of $M_0$ and $M_0'$ is determined by their geometry at infinity,
we have $\lambda_{\ess}(M_0)=\lambda_{\ess}(M_0')=\lambda_0(H)$.
Therefore \cref{name} applies to Riemannian coverings of $M_0$
if $\lambda_0(M_0)<\lambda_0(H)$.

Since $M_0'$ is geometrically finite and the volume $|M_0'|=\infty$,
the boundary $\partial C$ of the convex core $C$ of $M_0'$ is non-empty.
Let $D\subset\partial C$ be any (small) simply connected domain,
and let $E$ be the set of points in $M_0'\setminus C$ such that any $x\in E$
has a minimal geodesic connection to $C$ with tip in $D$.
Then $E$ contracts onto $D$ and grows exponentially with the distance to $D$.
Hence given any $r>0$,
there is a ball of radius $r$ in $E$, which is isometric to a ball of radius $r$ in $H$,
that does not intersect the $r$-neighborhood of $C$ and $K'$.
Such a ball has its sibling in $M_0$, and hence the covering space $M_1$ also contains such balls.
Therefore $\lambda_0(M_1)\le\lambda_0(H)$,
and hence $\lambda_0(M_1)=\lambda_0(M_0)$, by monotonicity.
\end{proof}

\appendix
\section{Remarks on differential operators}
\label{frido}

Let $E$ be a vector bundle (real or complex) over a Riemannian manifold $M$.
Albeit inconsistent, it will be convenient
to denote the space of smooth sections of $E$ by $C^\infty(E)$ or,
if need arises, by $C^\infty(M,E)$.
Similar conventions will be in force for other spaces of sections of $E$.

Assume that $E$ is endowed with a Riemannian metric and a compatible connection.
Let $\mu=\vf^2\dv$ be a weighted measure on $M$,
where $\vf\in C^\infty(M)$ is strictly positive.
In this appendix, we use the shorthands
\begin{align*}
	\la u,v\ra_\mu = \la u,v\ra_{L^2(E,\mu)}
	\quad\text{and}\quad
	\|u\|_\mu = \|u\|_{L^2(E,\mu)}
\end{align*}
for $L^2$-products and norms with respect to $\mu$.

Let $L$ be a differential operator of order $k$ on (smooth sections of) $E$.
Then there is a unique differential operator $L^{\ad}$ of order $k$ on $E$,
the \emph{formal adjoint of $L$ (with respect to $\mu$)} such that
\begin{align*}
	\la Lu,v\ra_\mu = \la u,L^{\ad}v\ra_\mu
\end{align*}
for all $u,v\in C^\infty(E)$ such that $\supp u\cap\supp v$ is compact.
We say that $L$ is \emph{formally self-adjoint (with respect to $\mu$)} if $L=L^{\ad}$.
We say that \emph{$L$ is bounded from below (with respect to $\mu$)} if
\begin{align*}
	\la Lu,u\ra_\mu \ge \beta\|u\|_\mu^2
\end{align*}
for all $u\in C^\infty_c(E)$.
Then we call $\beta$ a \emph{lower bound of $L$ (with respect to $\mu$)} and write $L\ge\beta$.
We say that $L$ is \emph{non-negative (with respect to $\mu$)} if $L\ge0$.
Clearly, $L\ge\beta$ if and only if $L-\beta\ge0$.

The \emph{principal symbol of $L$}, frequently written in terms of local coordinates,
can also be written as
\begin{align}\label{psymbol}
	\sigma_L(df)u
	= \frac1{k!} [\dots [L,\underbrace{m_f],\dots m_f]}_{\text{$k$ times}} u,
\end{align}
where $f\in C^\infty(M)$, $u\in C^\infty(E)$, and $m_f$ denotes multiplication with $f$.
Recall that $L$ is of \emph{Laplace type} if it is of order two and its principal symbol is given by
\begin{align}\label{lapty}
	\sigma_L(df)u = \frac1{2}[[L,m_f],m_f]u = - |\grad f|^2u 
\end{align}
for all $f\in C^\infty(M)$ and $u\in C^\infty(E)$.
Laplace type operators are elliptic.

Denoting the connection on $E$ by $\nabla$ (as any other connection),
the formal adjoint $\nabla^{\ad}$ of $\nabla$ is given by
\begin{align}\label{forad}
\begin{split}
	\nabla^{\ad}v
	&= - \sum \{ (\nabla_{X_i}v)(X_i) + 2X_i(\ln\vf)v(X_i) \} \\
	&= - \dive v - 2v(\grad\ln\vf)
\end{split}
\end{align}
for any $v\in C^\infty(T^*M\otimes E)$,
where $(X_i)$ is a local orthonormal frame of $TM$.
Clearly, the \emph{Bochner-Laplacian (associated to $\nabla$ and $\mu$)}, given by
\begin{align}\label{lapty2}
	\Delta_\mu = \nabla^{\ad}\nabla,
\end{align}
is an operator of Laplace type on $E$.
More generally,
we will also study differential operators of \emph{generalized Laplace type},
that is, $L$ is elliptic and of the form
\begin{align}\label{diver}
	L = A^{\ad}A + B,
\end{align}
where $A$ is of order one and $B$ of order at most one.
By definition, $A^{\ad}A$ is formally self-adjoint.
More precisely, we have
\begin{align}\label{diver2}
  \la A^{\ad}Au,v \ra_\mu
  = \la Au,Av \ra_\mu
  = \la u,A^{\ad}Av\ \ra_\mu
\end{align}
for all $u,v\in C^\infty(E)$ such that $\supp u\cap\supp v$ is compact.
Obviously, $L$ is formally self-adjoint if and only $B$ is.

\begin{lem}
If a differential operator $B$ of order one is formally self-adjoint,
then $\sigma_B(\omega)$ is a skew-symmetric (resp.\ skew-Hermitian) field of endomorphisms of $E$,
for any one-form $\omega$ on $M$.
\end{lem}

\begin{proof}
For any $f\in C^\infty_c(M)$,
multiplication $m_f$ with $f$ is a symmetric (resp.\ Hermitian) operator on $L^2(E,\mu)$.
Hence the commutator $\sigma_B(df)=[B,m_f]$ is a skew-symmetric (resp.\ skew-Hermitian) field
of endomorphisms of $E$ if $B$ is formally self-adjoint.
\end{proof}

\begin{exa}\label{exalop}
The  \emph{Hodge-Laplacian} $(d+d^{\ad})^2=dd^{\ad}+d^{\ad}d$
on the bundle of differential forms over $M$ is a Laplace type operator
(where the formal adjoint $d^{\ad}$ of $d$ is taken with respect to $\mu=\dv$).
More generally,
the square $A^2$ of the Dirac operator $A$ on a Dirac bundle over $M$ is a Laplace type operator.
Since $A$ is formally self-adjoint with respect to $\dv$,
$A^2$ is of generalized Laplace type with $B=0$ (see \eqref{diver}).
By \eqref{diver2},
$A^2$ is formally self-adjoint and non-negative with respect to $\dv$.
\end{exa}

\begin{exa}\label{exasop}
An operator on $E$ of the form $S=\Delta_\mu+V$,
where the \emph{potential} $V$ is a field of endomorphisms of $E$,
is called a \emph{Schr\"odinger operator (with respect to $\mu$)}.
It is formally self-adjoint if and only if its potential is a symmetric (resp. Hermitian) field of endomorphisms of $E$.
If $V\ge0$, then $S$ is non-negative.
Hodge-Laplacians and, more generally, squares of Dirac operators,
are Schr\"odinger operators with respect to $\dv$,
where the potential is a curvature term.
\end{exa}

\subsection{Friedrichs extension}
\label{susfri}
Let $L$ be a differential operator of order $k$ on $E$
which is formally self-adjoint and bounded from below with respect to the weighted measure $\mu=\vf^2\dv$.

Choose a lower bound $\beta$ for $L$,
and let $H_L$ be the completion of $C^\infty_c(E)$ with respect to the product
\begin{align}\label{h1}
	\la u,v \ra_L = \la u,v \ra_\mu + \la(L-\beta)u,v \ra_\mu.
\end{align}
We consider $H_L$ as a subspace of $H=L^2(E,\mu)$,
endowed with its own, stronger inner product.
Up to equivalence, $\la.,.\ra_L$ does not depend on the choice of $\beta$,
and hence $H_L$ does not depend on the choice of $\beta$ either.

\begin{thm}\label{friex}
If $L\ge\beta$, then the Friedrichs extension $\bar L$ of $L$
is a nonnegative self-adjoint extension of $L$ with $\bar L=L^*$ on its domain
\begin{align*}
	D(\bar L)
	= H_L \cap D(L^*) \subseteq L^2(E,\mu).
\end{align*}
Furthermore,
\begin{enumerate}
\item\label{friexh1}
for any $u\in D(\bar L)$,
we have $\la u,v\ra_L = \la u,v\ra_\mu + \la(\bar L-\beta)u,v\ra_\mu$ for all $v\in H_L$;
\item\label{friexes}
if $L^*-\lambda$ is injective for some $\lambda<\beta$, then $L$ is essentially self-adjoint,
and its closure coincides with $\bar L$.
\end{enumerate}
\end{thm} 

We write $\sigma(L,M)$ for the spectrum and $\lambda_0(L,M)$ for the bottom of the spectrum of $\bar L$.
For a non-vanishing $u\in D(\bar L)$, we denote by
\begin{align}\label{ray}
  R(u) = R_L(u) = \la\bar Lu,u\ra_\mu/\la u,u\ra_\mu = \|u\|_L^2/\|u\|_\mu + \beta -1
\end{align}
its \emph{Rayleigh quotient}.
By functional analysis, we have $\lambda_0(L,M)=\inf R(u)$,
where the infimum is taken over all non-zero $u\in D(\bar L)$.

\begin{cor}
If $L\ge\beta$, then $\lambda_0(L,E)=\inf R(u)$,
where the infimum is taken over all non-zero $u\in C^\infty_c(E)$.
\end{cor}

\begin{proof}
Let $u\in D(\bar L)$ be non-zero.
Choose a sequence $(u_n)$ in $C^\infty_c(E)$ converging to $u$ with respect to the $H_L$-norm.
Then $R(u)$ is the limit of the sequence of $R(u_n)$, by \eqref{ray}.
\end{proof}

\subsection{Geometric Weyl sequences}
\label{suseweyl}
The essential spectrum $\sigma_{\ess}(A)$ of a self-adjoint operator $A$ on a Hilbert space $H$
consists of all $\lambda\in\R$ such that $A-\lambda\id$ is not a Fredholm operator.
The essential spectrum is a closed subset of the spectrum $\sigma(A)$ of $A$,
and its complement in $\sigma(A)$ consists of isolated eigenvalues of finite multiplicity.
By \emph{Weyl's Criterion}, $\lambda\in\R$ belongs to $\sigma_{\ess}(A)$ if and only if
there exists a sequence $(u_n)$ in the domain $D(A)$ such that, as $n\to\infty$,
\begin{enumerate}
\item
$\|u_n\|_H\to1$;
\item
$u_n\rightharpoonup0$ in $H$;
\item
$\|(A-\lambda)u_n\|_H\to 0$.
\end{enumerate}
Here $\|.\|_H$ denotes the norm of $H$.
Such a sequence is also called a \emph{Weyl sequence} for $\lambda$ or, more precisely, for $A$ and $\lambda$.

We let $L$ now be an elliptic differential operator of order $k$ on $E$
which is formally self-adjoint and bounded from below with respect to the weighted measure $\mu=\vf^2\dv$
and denote by $\bar L$ the Friedrichs extension of $L$.

For a relatively compact open domain $U\subseteq M$ with smooth boundary,
we denote by $H^k(U,E)$ the Sobolev space of sections of $E$ over $U$ with weak derivatives of order up to $k$ in $L^2(U,E)$.
We do not include $\mu$ into the notation of $H^k(U,E)$ since, on $U$,
the weight $\vf^2$ of $\mu$ is bounded between two positive constants.

\begin{thm}\label{weyl}
Assume that $L$ is formally self-adjoint and bounded from below.
Then $\lambda\in\R$ belongs to the essential spectrum $\sigma_{\ess}(L,M)$ of $\bar L$
if and only if there is a \emph{geometric Weyl sequence for $\lambda$},
that is, a sequence $(u_n)$ in $D(\bar L)$ such that
\begin{enumerate}
\item 
$\|u_n\|_\mu\to1$ as $n\to\infty$;
\item
for any compact subset $K\subseteq M$, $\supp u_n\subseteq M\setminus K$ eventually;
\item
$\|(\bar L-\lambda)u_n\|_\mu\to 0$ as $n\to\infty$.
\end{enumerate}
\end{thm}

\begin{proof}
Clearly, any geometric Weyl sequence $(u_n)$ converges weakly to $0$ in $H=L^2(E,\mu)$
and hence is a Weyl sequence for $\lambda$ as defined above.
Applying Weyl's Criterion, we conclude that $\lambda\in\sigma_{\ess}(L,M)$.

Conversely, let $\lambda\in\sigma_{\ess}(L,M)$ and $(u_n)$ be a Weyl sequence for $\lambda$.
Let $\psi\in C^{\infty}_{c}(M)$.
The main step of the proof consists in showing that,
after passing to a subsequence if necessary,
$((1-\psi)u_n)$ is a Weyl sequence for $\lambda$. 

Weak convergence to zero in $L^2(E,\mu)$ is easy to see:
For any $v\in L^2(E,\mu)$,
we have $\psi v,(1-\psi)v\in L^2(E,\mu)$ and
\begin{align*}
  \la \psi u_n,v\ra_\mu
  = \la u_n,\psi v\ra_\mu \to 0
  \quad\text{and}\quad
  \la (1-\psi)u_n,v\ra_\mu
  = \la u_n,(1-\psi) v\ra \to 0
\end{align*}
since $u_n\rightharpoonup0$ in $L^2(E,\mu)$.
Hence $\psi u_n\rightharpoonup0$ and $(1-\psi)u_n\rightharpoonup0$ in $L^2(E,\mu)$.

Choose relatively compact open domains $U\subset\subset V\subseteq M$ with smooth boundary containing $\supp\psi$. 
By the ellipticity of $L$, there is a constant $C_1$ such that
\begin{align*}
	\|u\|_{H^k(U,E)} \le C_1( \|\bar Lu\|_{L^2(V,E)} + \|u\|_{L^2(V,E)})
\end{align*}
for all $u\in H^k(V,E)$.
Furthermore, (the restriction to $V$ of) any $u\in D(\bar L)$ belongs to $H^k(V,E)$.
Since
\begin{align*}
	\|\bar Lu_n\|_\mu
	\le \|(\bar L-\lambda)u_n\|_\mu + \lambda \|u_n\|_\mu
	\le (\lambda+1) \|u_n\|_\mu
\end{align*}
for all sufficiently large $n$,
we conclude that $(u_n)$ is uniformly bounded in $H^k(U,E)$.
Hence there is a $u\in H^k(U,E)$ such that,
up to passing to a subsequence if necessary, $u_n\rightharpoonup u$ in $H^k(U,E)$.
Now $u_n\rightharpoonup 0$ in $L^2(E,\mu)$ and hence $u=0$.
Since $U$ is relatively compact,
the Rellich Lemma applies and shows that $\|u_n\|_{H^{k-1}(U,E)}\to 0$.
In particular, $\|\psi u_n\|_\mu\to0$ and, therefore, $\|(1-\psi)u_n\|_\mu\to1$.

Since $u_n\in H^k(U,E)$,
there is a sequence $v_{n,l}\in C^\infty(U,E)$ converging to $u_n$ in $H^k(U,E)$ as $l\to\infty$.
But then $\psi v_{n,l}$ converges to $\psi u_n$ in $H^k(U,E)$ as $l\to\infty$.
In particular, $\psi u_n\in D(\bar L)$ and
\begin{align*}
	\bar L(\psi u_n)
	= \psi\bar Lu_n + [\bar L,m_\psi]u_n.
\end{align*}
Here the commutator $[\bar L,m_\psi]$ is a differential operator of order at most $k-1\ge0$
with support contained in $\supp d\psi\subseteq U$.
We conclude that $(1-\psi)u_n\in D(\bar L)$ with
\begin{align*}
	\bar L((1-\psi)u_n)
	= (1-\psi)\bar Lu_n - [\bar L,m_\psi]u_n.
\end{align*}
Since $K=\supp d\psi\subseteq U$ is compact, there is a constant $C_2$ such that 
\begin{align*}
	\|[\bar L,m_\psi]u\|_{L^2(K,E)} \le C_2 \|u\|_{H^{k-1}(U,E)}
\end{align*}
for all $u\in H^{k-1}(U,E)$.
In conclusion,
\begin{align*}
	\|(\bar L-\lambda)((1&-\psi)u_n)\|_\mu \\
	&\le \|(1-\psi)(\bar L-\lambda)u_n\|_\mu + C_3\|[\bar L,m_\psi]u_n\|_{L^2(K,E)} \\
	&\le \|(\bar L-\lambda)u_n\|_\mu + C_2C_3\|u_n\|_{H^{k-1}(U,E)}
	\to 0,
\end{align*}
where $C_3$ estimates $\dv$ against $\mu$.
Now we choose a sequence of functions $0\le\psi_1\le\psi_2\le\dots\le1$ in $C^\infty_c(E)$
such that the compact subsets $\{\psi_n=1\}$ exhaust $M$ and obtain that,
after passing to appropriate subsequences and multiplications by cut-off functions $1-\psi_n$ consecutively,
we obtain a geometric Weyl sequence for $\lambda$.
\end{proof}

\begin{rem}\label{weylesa}
In the case where $L$ is essentially self-adjoint,
geometric Weyl sequences can be chosen to belong to $C^\infty_c(E)$.
\end{rem}

\begin{cor}\label{weylmk}
If $L$ is formally self-adjoint and bounded from below by $\beta$ on $E$ over $M$ and $K\subseteq M$ is compact,
then $L$ is formally self-adjoint and bounded from below by $\beta$ on $E$ over $M\setminus K$ and
\begin{align*}
	\sigma_{\ess}(L,M) \subseteq \sigma_{\ess}(L,M\setminus K).
\end{align*}
In particular, $\lambda_{\ess}(L,M)\ge\lambda_{\ess}(L,M\setminus K)$.
\end{cor}

\begin{proof}
Let $(u_n)$ be a geometric Weyl sequence for $\lambda\in\sigma_{\ess}(L,M)$.
Then $\supp u_n\subseteq M\setminus K$ for all sufficiently large $n$.
For such $n$, $u_n$ belongs to the domain of the Friedrichs extension of $L$ on $M\setminus K$.
Thus these $u_n$ constitute  a geometric Weyl sequence for the Friedrichs extension of $L$ on $M\setminus K$.
Thus $\sigma_{\ess}(L,M)\subseteq\sigma_{\ess}(L,M\setminus K)$.
\end{proof}

\subsection{Stability of the essential spectrum}
\label{susestab}
One might ask whether equality of essential spectra holds in \cref{weylmk}.
The problem is that the intersection of a compact subset of $M$ with $M\setminus K$ need not be compact in $M\setminus K$.
However, if $K$ is sufficiently regular, equality holds.
We show this for operators of generalized Laplace type (as in \eqref{diver}).
However, it seems that the same proof, with a bit more of technicalities,
goes through for elliptic operators as considered above.

\begin{thm}[Stability of the essential spectrum]\label{esstab}
Suppose that $L$ is of generalized Laplace type, $L=A^{\ad}A+B$.
Assume that $L$ is formally self-adjoint and bounded from below.
Let $K\varsubsetneq M$ be a compact domain with smooth boundary.
Then
\begin{align*}
	\sigma_{\ess}(L,M) = \sigma_{\ess}(L,M\setminus K).
\end{align*}
\end{thm}

We assume throughout and without loss of generality that $L\ge0$.
We start with some preparatory steps. 

Let $\Omega\subseteq M$ be an open domain with compact and smooth boundary.
Let $C^\infty_{c,0}(\bar\Omega,E)$ be the space of $u\in C^\infty_c(\bar\Omega)$
such that $u$ satisfies the Dirichlet boundary condition $u=0$ on $\partial\Omega$.
Then $C^\infty_c(\Omega,E)$ is contained in $C^\infty_{c,0}(\bar\Omega,E)$,
and $L$ is symmetric on $C^\infty_{c,0}(\bar\Omega,E)$.

\begin{lem}\label{hl}
The closure $H_L(\Omega)$ of $C^\infty_c(\Omega,E)$ under the $\|.\|_L$-norm
contains $C^\infty_{c,0}(\bar\Omega,E)$.
\end{lem}

\begin{proof}
Let $u\in C^\infty_{c,0}(\bar\Omega,E)$,
and let $C_A$ and $C_B$ be bounds for $|\sigma_A|$ and $|\sigma_B|$ on $\supp u$.
Let $\chi_n\in C^\infty_c(\bar\Omega)$ be a sequence of functions
with $0\le\chi_n\le1$, $\chi_n(x)=0$ if $d(x,\partial\Omega)\le1/n$,
$\chi_n(x)=1$ if $d(x,\partial\Omega)\ge2/n$, and $|d\chi_n|\le 2n$.

Since $u$ is smooth and vanishes along $\partial\Omega$,
there is a constant $C_1>0$ such that $|u(x)|\le C_1d(x,\partial\Omega)$.
Hence $|\sigma_A(d\chi_n)u|\le 4|\sigma_A|C_1\le4C_AC_1$ and therefore
\begin{align*}
  \|\sigma_A(d\chi_n)u\|_\mu^2
  = \int_{\Omega} |\sigma_A(d\chi_n)u|^2d\mu
  \le 16C_A^2C_1^2C_2/n,
\end{align*}
where $C_2>0$ is a constant such that
\begin{align*}
	|\{ x\in\Omega\mid 1/n\le d(x,\partial\Omega)\le2/n\}|_\mu\le C_2/n.
\end{align*}
Similarly,
\begin{align*}
  \|\sigma_B(d\chi_n)u\|_\mu^2
  = \int_{\Omega} |\sigma_B(d\chi_n)u|^2d\mu
  \le 16C_B^2C_1^2C_2/n.
\end{align*}
We have $\chi_nu\in C^\infty_c(\Omega,E)$ and get
\begin{align*}
    \la L((1-\chi_n)&u),(1-\chi_n)u \ra_\mu \\
    &\le \|A((1-\chi_n)u)\|_\mu^2 + \|B((1-\chi_n)u)\|_\mu\|(1-\chi_n)u\|_\mu \\
    &\le \|(1-\chi_n)Au-\sigma_A(d\chi_n)u\|_\mu^2 \\
    &\hspace{11mm} + \|(1-\chi_n)Bu-\sigma_B(d\chi_n)u\|_\mu\|(1-\chi_n)u\|_\mu \\
    &\le 2\|(1-\chi_n)Au\|_\mu^2 + 2\|\sigma_A(d\chi_n)u\|_\mu^2 \\
    &\hspace{5mm} + (\|(1-\chi_n)Bu\|_\mu + \|\sigma_B(d\chi_n)u\|_\mu)\|(1-\chi_n)u\|_\mu,
\end{align*}
where we use $\la.,.\ra_\mu$ and $\|.\|_\mu$ to indicate corresponding integrals against $d\mu$ over $\Omega$.
Now the right hand side tends to zero.
Therefore $u=\lim(\chi_nu)$ with respect to $\|.\|_L$,
and hence $u$ lies in the $\|.\|_L$-closure of $C^\infty_c(\Omega,E)$.
\end{proof}

To avoid confusion, we denote $L$ with domain $C^\infty_c(\Omega,E)$ by $L_c$
and $L$ with domain $C^\infty_{c,0}(\bar\Omega,E)$ by $L_{c,0}$.
Since $C^\infty_{c,0}(\bar\Omega,E)$ is contained in the domain of the adjoint $L_c^*$ of $L_c$ and $L_c^*=L_{c,0}$ on $C^\infty_{c,0}(\bar\Omega,E)$,
we get that $L_c\subseteq L_{c,0}$ and, using \cref{hl},
that the Friedrichs extension $\bar L_c$ contains $L_{c,0}$.
In particular, $L_{c,0}$ is non-negative on $C^\infty_{c,0}(\bar\Omega,E)$
and hence the Friedrichs extension $\bar L_{c,0}$ of $L_{c,0}$ is defined.

\begin{prop}\label{ll}
We have $\bar L_c=\bar L_{c,0}$.
\end{prop}

\begin{proof}
By definition and \cref{hl}, the domains of $\bar L_c$ and $\bar L_{c,0}$ are
\begin{align*}
  H_L(\Omega)\cap D(L_c^*)
  \quad\text{and}\quad
  H_L(\Omega)\cap D(L_{c,0}^*)
\end{align*}
Since $L_c\subseteq L_{c,0}$, we have $L_{c,0}^*\subseteq L_c^*$.
Now $\bar L_c$ and $\bar L_{c,0}$ are self-adjoint, and hence they coincide.
\end{proof}

\begin{proof}[Proof of \cref{esstab}]
By \cref{weylmk}, it remains to prove that
\begin{align*}
	\sigma_{\ess}(L,M \setminus K) \subseteq \sigma_{\ess}(L,M).
\end{align*}
To this end, notice that $\Omega=M\setminus K$ is an open domain in $M$ with compact smooth boundary.
Hence, with $L_c$ and $L_{c,0}$ as above, \cref{ll} implies that
$\sigma_{\ess}(L_c,M\setminus K) = \sigma_{\ess}(L_{c,0},M\setminus\mathring{K})$.
From Weyl's criterion,
we know that, for any $\lambda\in\sigma_{\ess}(L,M\setminus \mathring{K})$,
there exists a Weyl sequence $(u_n)_{n\in\N}$ for $\bar L_{c,0}$ and $\lambda$.
	
Let $U\subset\subset V$ be relatively compact open neighborhoods of $K$ in $M$ with smooth boundary.
Consider $\chi\in C^{\infty}_{c}(M)$ with $\chi =1$ on $U$ and $\supp\chi\subseteq V$.
Since the Dirichlet boundary condition is elliptic, we have an estimate
\begin{align}
	\|u\|_{H^2(\Omega\cap U,E)} \le C( \|Lu\|_{L^2(\Omega\cap V,E)} + \|u\|_{L^2(\Omega\cap V,E)})
\end{align}
for all $u\in H^2(\Omega\cap V,E)$ which vanish on $\partial K$. Arguing as in the proof of \cref{weyl},
we now get that the sections $(1 - \chi)u_n$ form a Weyl sequence for $\bar L_{c,0}$ and $\lambda$.
Since the supports of the $(1-\chi)u_n$ are contained in $M\setminus U$,
they also form a Weyl sequence for $\bar L$ and $\lambda$.
\end{proof}

\begin{cor}[Decomposition principle]\label{decpri}
For $i=1,2$,
let $L_i$ be operators of generalized Laplace type on vector bundles $E_i$ over Riemannian manifolds $M_i$
which are formally self-adjoint with respect to weighted measures $\mu_i$ and are bounded from below.
Assume that, for some compact domains $K_i\subseteq M_i$ with smooth boundary,
there is an isometry $M_1\setminus K_1\to M_2\setminus K_2$
which preserves the weighted measures
and transforms $E_2$ to $E_1$ and $L_2$ to $L_1$ over these domains.
Then $\sigma_{\ess}(L_1,M_1)=\sigma_{\ess}(L_2,M_2)$.
\end{cor}

\subsection{Essential self-adjointness}
\label{susesa}
The following discussion was motivated by \cite[Lecture 2]{Ba},
where scalar diffusion operators on Euclidean spaces are considered,
but our arguments and presentation changed with time.
Our main result, \cref{lesa4}, is known, at least in the scalar case;
compare e.g. with \cite{Ch} and references therein.
However, our discussion is quite elementary and short and might therefore be welcome.

We consider elliptic operators of generalized Laplace type as in \eqref{diver},
which are formally self-adjoint with respect to a weighted measure $\mu$
and assume that $L$ is bounded from below by $\beta\in\R$.

In what follows, we do not assume throughout that sections of $E$ are square-integrable,
and therefore we use explicit integrals instead of the shorter $L^2$-product notation
in corresponding places of our computations.

\begin{lem}\label{lesa}
If $u\in C^\infty(E)$ solves $Lu=\lambda u$ for some $\lambda\in\R$, then
\begin{align*}
	\int_M \la Au,A(f^2u)\ra d\mu
	\le \|A(fu)\|_\mu^2 + (\lambda-\beta)\|fu\|_\mu^2
\end{align*}
for any $f\in C^\infty_c(M)$.
\end{lem}

\begin{proof}
For $f\in C^\infty_c(M)$, we have $f u,f^2u\in C^\infty_c(E)$,
and hence integrations by parts below do not lead to boundary terms.
We have
\begin{align*}
	\int_M \la Au,A(f^2u)\ra d\mu
	&= \int_M \la u,(L-B)(f^2u)\ra d\mu \\
	&= \int_M \la Lu,f^2u\ra d\mu - \int_M \la u,B(f^2u)\ra d\mu \\
	&= \lambda \int_M \la fu,fu\ra d\mu - \int_M \la fu,B(fu)\ra d\mu \\
	&\le \|A(fu)\|_\mu^2 + (\lambda-\beta)\|fu\|_\mu^2,
\end{align*}
where we use, in the penultimate step, that $\sigma_B(df)$ is a field of skew-symmetric (resp.\ skew-Hermitian) endomorphisms of $E$
\end{proof}

\begin{lem}\label{lesa2}
If $u\in C^\infty(E)$ solves $Lu=\lambda u$ for some $\lambda\in\R$, then
\begin{align*}
	(\beta-\lambda)\|fu\|_\mu^2 \le \|\sigma_A(df)u\|_\mu^2
\end{align*}
for any $f\in C^\infty_c(M)$.
\end{lem}

\begin{proof}
Since
\begin{align*}
	|A(fu)|^2 = |fAu+\sigma_A(df)u|^2 = \la Au,A(f^2u)\ra + |\sigma_A(df)u|^2,
\end{align*}
we obtain from \cref{lesa} that
\begin{align*}
	(\beta-\lambda)\|fu\|_\mu^2
	\le \|A(fu)\|_\mu^2 - \int_M \la Au,A(f^2u)\ra d\mu
	= \|\sigma_A(df)u\|_\mu^2,
\end{align*}
which is the asserted inequality.
\end{proof}


\begin{thm}\label{lesa4}
Suppose that $M$ is complete, that $L$ is bounded from below, and that $\int_0^\infty dr/s(r)=\infty$,
where $r$ denotes the distance to some fixed point in $M$ and $s(t)=\max_{\{r(x)\le t\}}|\sigma_A|_x$.
Then $L$ is essentially self-adjoint,
and its closure coincides with the Friedrichs extension $\bar L$ of $L$.
\end{thm}

\begin{proof}
We consider the case where $\sigma_A$ is bounded first.
Let $\beta$ be a lower bound for $L$, and choose $\lambda<\beta$.
Suppose that there is a $u\in D(L^*)$ with $L^*u=\lambda u$.
Then $Lu=\lambda u$ weakly, by the symmetry of $L$.
Since $L$ is elliptic, this implies that $u\in C^\infty(E)$ with $Lu=\lambda u$.
 
Since $M$ is complete,
there is a sequence $0\le\vf_1\le\vf_2\le\dots\le1$ in $C^\infty_c(M)$ such that
\begin{enumerate}
\item\label{fn1}
$\{\vf_1=1\}\subseteq\{\vf_2=1\}\subseteq\dots$ exhausts $M$;
\item\label{fn2}
$\|d\vf_n\|_\infty\le1/n$.
\end{enumerate}
Then $\|\vf_nu\|_\mu\to\|u\|_\mu$ as $n\to\infty$.
Since $\|\sigma_A\|_\infty<\infty$ and $\|d\vf_n\|_\infty\to0$ as $n\to\infty$, 
\cref{lesa2} together with $\beta-\lambda>0$ implies that $u=0$.

We now use the trick from \cite[page 21]{BB} to reduce the more general case
to the case where $\sigma_A$ is bounded.
We choose a smooth function $f\colon M\to\R$ such that
\begin{align*}
  s(r(x)) < f(x) < s(r(x))+1.
\end{align*}
Then $M$ with the Riemannian metric $g'=f^{-2}g$ is also complete,
where $g$ denotes the original metric of $M$; see loc.\,cit.
The new volume element is $\mu'=f^{-m}\mu$, and the isomorphism
\begin{align*}
  L^2(E,\mu') \to L^2(E,\mu), \quad u\mapsto f^{-m/2}u,
\end{align*}
is orthogonal and transforms $L$ into an operator $L'=A^{\ad'}A+B'$ of generalized Laplace type,
where $A^{\ad'}$ denotes the formal adjoint of $A$ with respect to $\mu'$
and $B'$ is of order at most one.
With respect to $g'$, we have $|\alpha|'=f(x)|\alpha|$, where $\alpha\in T_x^*M$.
Therefore 
\begin{align*}
  |\sigma_A|'_x
  = \sup_{0\ne\alpha\in T_x^*M}\frac{|\sigma_A(\alpha)|}{|\alpha|'}
  = \frac1{f(x)} |\sigma_A|_x
  \le\frac{s(x)}{f(x)} \le 1,
\end{align*}
and hence $|\sigma_A|'_\infty\le1$.
We conclude from the first part of the proof that $L'$ is essentially self-adjoint,
and hence its transform $L$ is also essentially self-adjoint.
\end{proof}

Bochner-Laplacians, Hodge-Laplacians, and squares of Dirac operators on Dirac bundles
are of generalized Laplace type, non-negative, and with parallel, hence bounded $\sigma_A$. 
Therefore \cref{lesa4} applies to them.


\end{document}